\documentclass{amsart}
\usepackage{times,mathrsfs,array,amssymb,pb-diagram,epsfig}

\usepackage{footnote}
\usepackage[section]{placeins} 
\usepackage{hyperref} 

\newcounter{lemma}
\newtheorem{Theorem}{Theorem}
\newtheorem{Lemma}[lemma]{Lemma}
\newtheorem{Corollary}[lemma]{Corollary}
\newtheorem{Proposition}[lemma]{Proposition}

\theoremstyle{definition}
\newtheorem{Example}[lemma]{Example}
\newtheorem{Definition}[lemma]{Definition}

\newtheorem{Remark}[lemma]{Remark}

\def\C{\mathbb C}
\def\H{\mathbb H}
\def\O{\mathcal O}

\def\Q{\mathbb Q}
\def\R{\mathbb R}
\def\Z{\mathbb Z}

\def\Im{\mathrm{Im}\,}

\def\div{\operatorname{div}}
\def\mod{\  \mathrm{mod}\ }

\def\Im{\mathrm{Im\,}}
\def\tr{\operatorname{tr}}

\def\JS#1#2{\left(\frac{#1}{#2}\right)}
\def\GL{\mathrm{GL}}
\def\SL{\mathrm{SL}}
\def\CM{\mathrm{CM}}
\def\Gal{\mathrm{Gal}}
\def\M#1#2#3#4{\begin{pmatrix}#1&#2\\#3&#4\end{pmatrix}}
\def\SM#1#2#3#4{\left(\begin{smallmatrix}#1&#2\\#3&#4\end{smallmatrix}
  \right)}

\newcommand*{\rom}[1]{\expandafter\@slowromancap\romannumeral #1@}

\begin{document}

\title{Ramanujan-type $1/\pi$-series from bimodular forms}

\author{Liuquan Wang and Yifan Yang}
\address{School of Mathematics and Statistics, Wuhan University, Wuhan 430072, Hubei, People's Republic of China}
\address{Mathematics Division, National Center for Theoretical Sciences, Taipei 10617, Taiwan}
\email{wanglq@whu.edu.cn;mathlqwang@163.com}

\address{Department of Mathematics, National Taiwan University and
  National Center for Theoretical Sciences, Taipei 10617, Taiwan}
\email{yangyifan@ntu.edu.tw}



\begin{abstract}
We develop an approach to establish $1/\pi$-series from bimodular
forms. Utilizing this approach, we obtain new families of $2$-variable
$1/\pi$-series associated to Zagier's sporadic Ap\'ery-like sequences.\end{abstract}

\maketitle

\section{Introduction}

In 1914, Ramanujan \cite{Ramanujan} gave seventeen series
representations for $1/\pi$ of the form
$$
\sum_{n=0}^\infty\frac{(1/2)_n(a)_n(1-a)_n}{(n!)^3}
(An+B)C^n=\frac D\pi
$$
with $a\in\{1/2,1/3,1/4,1/6\}$, $A,B,C\in\Q$, and $D\in\overline\Q$,
such as
$$
\sum_{n=0}^\infty\frac{(1/2)_n^3}{(n!)^3}(6n+1)
\left(\frac14\right)^n=\frac4\pi,
$$
where $(a)_n$ is the Pochhammer symbol defined by $(a)_n=a(a+1)\ldots(a+n-1)$.
The first proof of such identities was given by Borwein and Borwein
\cite{Borwein}. Since then, mathematicians
\cite{Almkvist-Guillera,Berndt,Borwein2,Chan-Chan-Liu,
  Chan-Cooper,Chan-Liaw,Chan-Liaw-Tan,Chan-Tanigawa,
  Chan-Verrill,Chudnovsky,Cooper,CooperWZ,Zudilin} have generalized Ramanujan's
$1/\pi$-series and produced many series of similar nature with the
coefficients $(1/2)_n(a)_n(1-a)_n/(n!)^3$ replaced by binomial sums,
such as the Ap\'ery numbers
$$
\sum_{k=0}^n\binom nk^2\binom{n+k}k^2.
$$
(See \cite{survey} for a more thorough survey.)

In \cite{Sun-list,Sun-pub}, through some powerful insight, Z.-W. Sun
discovered (conjectured) many two-variable anaologues of Ramanujan's
$1/\pi$-series. For instance, if we let
$$
T_n(b,c)=\sum_{m=0}^{\lfloor n/2\rfloor}\binom n{2m}
\binom{2m}m b^{n-2m}c^m
$$
denote the coefficient of $x^n$ in the expansion of $(x^2+bx+c)^n$,
then one of Sun's conjectural formulas states that
\begin{equation} \label{equation: Sun's example}
\sum_{n=0}^\infty\frac{(1/2)_n^2}{(n!)^2}T_n(34,1)(30n+7)
\left(-\frac1{64}\right)^n
=\frac{12}\pi
\end{equation}
(\cite[Entry (I2)]{Sun-list}). Among Sun's conjectural formulas, three
families are of the form
$$
\sum_{n=0}^\infty\frac{(a)_n(1-a)_n}{(n!)^2}T_n(b,c)(An+B)C^n
=\frac D\pi,
$$
$a\in\{1/2,1/3,1/4\}$, $b,c,A,B,C\in\Q$, and $D\in\overline\Q$. They
were proved by Chan, Wan, and Zudilin \cite{CWZ}.
They used cleverly a link
between $T_n(b,c)$ and Legendre polynomials, theory of hypergeometric
functions, and a formula of Brafman \cite{Brafman} to convert standard
Ramanujan-type $1/\pi$-series to Sun's formulas. Using similar
techniques, Rogers and Straub \cite{Rogers-Straub} proved another
family of Sun's conjectural formulas. More series related to Legendre polynomials and hence $T_n(b,c)$ were discussed by Wan \cite{Wan}.

It turns out that Sun's series
\begin{equation} \label{equation: Sun's series}
  \sum_{n=0}^\infty\frac{(a)_n(1-a)_n}{(n!)^2}T_n(b,c)z^n
  =\sum_{n=0}^\infty\frac{(a)_n(1-a)_n}{(n!)^2}T_n(bz,cz^2)
\end{equation}
were first considered by Brafman \cite{Brafman} in 1950's. For
instance, one of Brafman's results, Equation (13) of \cite{Brafman},
states that
\begin{equation*}
  \begin{split}
&\sum_{n=0}^\infty\frac{(a)_n(1-a)_n}{(n!)^2}P_n(x)t^n \\
&\qquad\qquad={}_2F_1\left(a,1-a;1;\frac{1-t-\rho}2\right)
  {}_2F_1\left(a,1-a;1;\frac{1+t-\rho}2\right),
\end{split}
\end{equation*}
where $\rho=(1-2xt+t^2)^{1/2}$ and
\begin{equation} \label{equation: Legendre}
P_n(x)=\sum_{m=0}^n\binom nm\left(\frac{x-1}2\right)^m
\left(\frac{x+1}2\right)^{n-m}
\end{equation}
is the $n$th Legendre polynomial. Since the Legendre polynomial can
also be expressed as
$$
P_n(x)=\sum_{m=0}^{\lfloor n/2\rfloor}\binom{2m}m\binom n{2m}
\left(\frac{x^2-1}4\right)^mx^{n-2m}=T_n\left(x,\frac{x^2-1}4\right)
$$
(see \eqref{equation: combinatorial} below), Brafman's formula shows
that the series in \eqref{equation: Sun's series} is equal to a
product of two ${}_2F_1$-hypergeometric functions.

The same series \eqref{equation: Sun's series} also appeared earlier
in literature in a very different context.
In \cite{Lian-Yau}, Lian and Yau explored relations among mirror
maps and Picard-Fuchs differential equations for families of toric
Calabi-Yau threefolds, modular forms, and hypergeometric
functions. One family of toric Calabi-Yau threefolds they considered
consists of hypersurfaces of degree $24$ in the weighted projective
space $\mathbb P^4[1,1,2,8,12]$. In an earlier work
\cite[(A.36)]{Hosono-Klemm}, it was shown that the Picard-Fuchs
system for this family is given by
\begin{equation*}
  \begin{split}
    &\theta_x(\theta_x-2\theta_y)-x(\theta_x+1/6)(\theta_x+5/6),
    \\
    &\theta_y(\theta_y-2\theta_z)-y(2\theta_y-\theta_x+1)
    (2\theta_y-\theta_x), \\
    &\theta_z^2-z(2\theta_z-\theta_y+1)(2\theta_z-\theta_y),
  \end{split}
\end{equation*}
where for a variable $t=x,y,z$, we let $\theta_t=t\partial/\partial
t$. Letting $z\to 0$, one gets a new system
\begin{equation*}
  \begin{split}
    &\theta_x(\theta_x-2\theta_y)-x(\theta_x+1/6)(\theta_x+5/6),
    \\
    &\theta_y^2-y(2\theta_y-\theta_x+1)
    (2\theta_y-\theta_x),
  \end{split}
\end{equation*}
which coincides with the Picard-Fuchs system for the family of toric
$K3$ surfaces corresponding to hypersurfaces of degree $12$ in
$\mathbb P^3[1,1,4,6]$. Up to scalars, this system has a unique
solution holomorphic near $(x,y)=(0,0)$. It is easy to verify that this solution is given by
\begin{equation*} \label{equation: alternative Sun}
\sum_{n=0}^\infty\sum_{m=0}^{\lfloor n/2\rfloor}
\frac{(1/6)_n(5/6)_n}{(n!)^2}\binom n{2m}\binom{2m}mx^ny^m
=\sum_{n=0}^\infty\frac{(1/6)_n(5/6)_n}{(n!)^2}T_n(x,x^2y),
\end{equation*}
which is Sun's series \eqref{equation: Sun's series} for the case
$a=1/6$. Likewise, Sun's series for
$a=1/4,1/3,1/2$, when written in the same form as the left-hand side
of \eqref{equation: alternative Sun}, are solutions of the
Picard-Fuchs systems
\begin{equation} \label{equation: hypergeometric PDE}
  \begin{split}
    &\theta_x(\theta_x-2\theta_y)-x(\theta_x+a)(\theta_x+1-a),
    \\
    &\theta_y^2-y(2\theta_y-\theta_x+1)
    (2\theta_y-\theta_x)
  \end{split}
\end{equation}
for some families of $K3$ surfaces.

One remarkable feature of these Picard-Fuchs systems \eqref{equation:
  hypergeometric PDE} is that they admit parameterization by bimodular
forms, first described by \cite[Corollary
1.3]{Lian-Yau} (for the case $a=1/6$) and later elaborated in more
details in \cite{Yang-Yui}.

The notion of bimodular forms was first introduced by
Stienstra and Zagier \cite{Stienstra-Zagier}. However, since bimodular
forms are not frequently studied and there is no
universally used definition, here we take the liberty to define
bimodular forms as follows.

\begin{Definition} Let $\Gamma$ be a congruence subgroup of
$\SL(2,\Z)$ and $N(\Gamma)$ be its normalizer in $\SL(2,\R)$. For a
subgroup $G$ of $N(\Gamma)$ containing $\Gamma$, a character $\chi_1$
on $\Gamma$, and a character $\chi_2$ on $G$ such that
$\chi_2|_\Gamma=\chi_1^2$, we say a function $F:\H^2\to\C$ is a
\emph{bimodular form} of weight $k$ on $(\Gamma,G)$ with characters
$(\chi_1,\chi_2)$ if
\begin{enumerate}
  \item[(i)] $F$ is a modular form of weight $k$ with character
    $\chi_1$ in each of the two variables, i.e., for all $\gamma=\SM
    abcd\in\Gamma$,
    \begin{equation} \label{equation: bimodular def 1}
    F\left(\gamma\tau_1,\tau_2\right)
    =\chi_1(\gamma)(c\tau_1+d)^kF(\tau_1,\tau_2), \quad
    F\left(\tau_1,\gamma\tau_2\right)
    =\chi_1(\gamma)(c\tau_2+d)^kF(\tau_1,\tau_2),
    \end{equation}
    and
  \item[(ii)] for all $\gamma=\SM abcd\in G$,
    \begin{equation} \label{equation: bimodular def 2}
      F\left(\gamma\tau_1,\gamma\tau_2\right)
        =\chi_2(\gamma)(c\tau_1+d)^k(c\tau_2+d)^kF(\tau_1,\tau_2).
    \end{equation}
\end{enumerate}
Likewise, we say a meromorphic function $f:\H^2\to\C$ is a bimodular
function on $(\Gamma,G)$ if \eqref{equation: bimodular def 1} and
\eqref{equation: bimodular def 2} hold with $k=0$ and the characters
$\chi_1$ and $\chi_2$ are trivial. We let $K(\Gamma,G)$ denote the field of bimodular functions on $(\Gamma,G)$.
\end{Definition}

In the present article we will focus on partial differential equations satisfied by bimodular forms and their application to two-variable Ramanujan $1/\pi$-series and leave other aspects of bimodular forms to future study.

The connection between bimodular forms and partial differential equations is given by Theorem 2.1 of \cite{Yang-Yui}, which states that if $F$ is a bimodular form of weight $1$ and $x$ and $y$ are two algebraically independent bimodular functions on $(\Gamma,G)$, then $F$, as a function of $x$ and $y$, satisfies a system of partial differential equations with algebraic functions of $x$ and $y$ as coefficients. (We will recall the exact expressions for the coefficients in the proof of Proposition \ref{proposition: sporadic PDEs} below.) When the genus of $\Gamma$ is $0$, the field of $K(\Gamma,G)$, being a subfield of $K(\Gamma,\Gamma)$, is a rational function field of transcendence degree $2$ over $\C$, by a theorem of Castelnuovo. In such a case, if $x$ and $y$ are chosen such that they generate the field of bimodular functions on $(\Gamma,G)$, then all the coefficients in partial differential equations will be rational functions in $x$ and $y$. Hence, when we expand $F$ with respect to $x$ and $y$, the coefficients will satisfy certain recursive relations.

It turns out that all the bimodular functions $x(\tau_1,\tau_2)$, $y(\tau_1,\tau_2)$ and bimodular forms $F(\tau_1,\tau_2)$ in the two-variable Ramanujan $1/\pi$-series in literature are invariant under the action of interchanging $\tau_1$ and $\tau_2$. This additional symmetry enables us to produce many more two-variable $1/\pi$-series with rational arguments using CM discriminants of larger class numbers. For instance, the pair of congruence subgroups associated to the series in \eqref{equation: Sun's example} is $(\Gamma_0(4),\Gamma_0(4)+\SM1021)$. Using the symmetries from $\SM1021$ and the interchanging of variables, one can show that the values of $x$ and $y$ are rational numbers if $\tau_1$ and $\tau_2$ are two suitably chosen CM-points of discriminant $d$ whose class group is a Klein $4$-group. For example, the series in \eqref{equation: Sun's example} comes from the case of $d=-192$.

In this paper, we shall consider the case where $G/\Gamma$ is a Klein $4$-group and the genus of the modular curve $X(\Gamma)$ is $0$. Using bimodular properties and the additional symmetry from interchanging of variables, we produce many new two-variable $1/\pi$-series with rational arguments. Most of our examples are coming from imaginary quadratic orders having a class group isomorphic to the elementary $2$-group of order $8$. In some rare cases, the class groups can be isomorphic to the direct product of a cyclic group of order $4$ and a cyclic group of order $2$ instead (notably, cases of discriminants $-576$, $-819$, $-1008$, $-3627$, and $-3843$ for $(\Gamma,G)=(\Gamma_0(9),\Gamma_0(9)+\SM3{-2}9{-3},w_9)$).

Note that it is rather easy to construct bimodular forms $F$ and bimodular functions $x$ and $y$ that generate $K(\Gamma,G)$ when $\Gamma$ has genus $0$. It is not difficult either to determine the recursive relations for the coefficients in the expansion of $F$ in terms of $x$ and $y$. However, it is hard to predict which $F$, $x$, and $y$ will yield an expansion with nice coefficients (say, expressible as sums of binomial coefficients). Here we will rely on results from \cite{CWZ,Wan-Zudilin}. The two-variable $1/\pi$-series obtained will involve Zagier's sporadic Ap\'ery-like sequences \cite{Zagier-Apery}, which we review in the next section.

\begin{Remark}
After the first draft of the paper was completed, we discovered that
Z.-W. Sun \cite{Sun-2019} has also found (conjectured) most of the
$1/\pi$-series in our tables in Appendix \ref{section: sporadic data}
independently. He also observed that many of his conjectures are related to imaginary quadratic fields of class number $8$.
We have made a comparison of our series with those of Sun \cite{Sun-list,Sun-2019} (proved by \cite{CWZ}) and Chan, Wan and Zudilin \cite{CWZ}, and give the corresponding equation numbers in their papers when the series coincide.
\end{Remark}

\section{Bimodular forms associated to Ap\'ery-like sequences}
Let $a$, $b$, and $c$ be real numbers and $\{u_n\}$ be the sequence of
numbers defined recursively by $u_{-1}=0$, $u_0=1$, and
\begin{equation} \label{equation: un}
(n+1)^2u_{n+1}=(an^2+an+b)u_n-cn^2u_{n-1}
\end{equation}
for $n\ge 0$. The series $f(t)=\sum_{n=0}^\infty u_nt^n$ is a solution
of the differential equation
\begin{equation} \label{equation: DE for un}
(\theta^2-t(a\theta^2+a\theta+b)+ct^2(\theta+1)^2)f=0,
\quad \theta=t\frac d{dt}.
\end{equation}
(Up to a scalar, this is the unique solution holomorphic near $t=0$.)
The definition of such a sequence is motivated by the Ap\'ery numbers,
corresponding to the case $(a,b,c)=(11,3,-1)$, that first appeared in
Ap\'ery's proof of irrationality of zeta-values
\cite{Apery,vanderPoorten}. A surprising feature
of the Ap\'ery numbers is that they are all integers. In literature,
we say $\{u_n\}$ is an \emph{Ap\'ery-like sequence} if all numbers
$u_n$ are integers. In \cite{Zagier-Apery}, Zagier did an extensive
search and found $36$ Ap\'ery-like sequences, up to scalings. Among
them, some are terminating sequences, some are polynomials (meaning
that $u_n=g(n)$ for some polynomial $g$), some are hypergeometric
(the cases where $c=0$), some are Legendrian (the cases where
$a^2=4c$), where the differential equation in \eqref{equation: DE for
  un} has three singularities and can be reduced to a hypergeometric
one, and there are six sporadic cases. All hypergeometric, Legendrian,
and sporadic cases have modular-function origin. Here we
are interested in the sporadic cases. Their modular properties and
corresponding $u_n$ are given in Table \ref{table: sporadic}.
\begin{table}[!htbp]
$$ \extrarowheight3pt\small
\begin{array}{lllll} \hline\hline
  (a,b,c) & \text{group} & t & f  & u_n \\ \hline
  (7,2,-8) & \Gamma_0(6) & \displaystyle\frac{1^36^9}{2^33^9}
           &\displaystyle\frac{2^13^6}{1^26^3}
                                 &\displaystyle\sum_{k=0}^n\binom nk^3 \\
  (10,3,9) & \Gamma_0(6) & \displaystyle\frac{1^46^8}{2^83^4}
           &\displaystyle\frac{2^63^1}{1^36^2}
           &\displaystyle\sum_{k=0}^n\binom nk^2\binom{2k}k \\
  (-17,-6,72)&\Gamma_0(6) &\displaystyle\frac{2^16^5}{1^53^1}
                         &\displaystyle\frac{1^66^1}{2^33^2}
    &\displaystyle\sum_{k=0}^n(-8)^{n-k}\binom nk
     \sum_{j=0}^k\binom kj^3 \\
  (12,4,32)&\Gamma_0(8)
                         &\displaystyle\frac{1^44^28^4}{2^{10}}
                         &\displaystyle\frac{2^{10}}{1^44^4}
    &\displaystyle\sum_{k=0}^n\binom nk\binom{2k}k\binom{2(n-k)}{n-k}\\
  (-9,-3,27) & \Gamma_0(9) & \displaystyle\frac{9^3}{1^3}
            &\displaystyle\frac{1^3}{3^1} & \displaystyle
    \sum_{k=0}^{\lfloor n/3\rfloor}(-3)^{n-3k}
    \binom n{3k}\frac{(3k)!}{(k!)^3} \\
  (11,3,-1) & \Gamma_1(5) & \text{see \eqref{add-t-defn}} & \text{see \eqref{add-f-defn}} &
  \displaystyle\sum_{k=0}^n\binom nk^2\binom{n+k}n\\
\hline\hline
\end{array}
$$
  \caption{Sporadic Ap\'ery-like sequences}
   \label{table: sporadic}
\end{table}
Here the notation $1^36^9/2^33^9$ etc. is a shorthand for
$\eta(\tau)^3\eta(6\tau)^9/\eta(2\tau)^3\eta(3\tau)^9$ etc. The
modular function $t$ and the modular form $f$ in the case $(11,3,-1)$
are
\begin{align}\label{add-t-defn}
t=q\prod_{n=1}^\infty(1-q^n)^{5\JS n5},
\end{align}
and
\begin{align}\label{add-f-defn}
f=\prod_{n=1}^\infty(1-q^{5n})^2\cdot\prod_{n>0,n\equiv1,4\mod
  5}(1-q^n)^{-3}\prod_{n>0,n\equiv 2,3\mod 5}(1-q^n)^2,
\end{align}
respectively.

\begin{Remark}
  Note that the congruence subgroups in all six cases have
  genus $0$, $4$ cusps, and no elliptic points.

  Note also that in \cite{Wan-Zudilin} and \cite{Zagier-Apery}, the
  case $(a,b,c)=(-9,-3,27)$ is listed as $(9,3,27)$ instead. The
  reason is that Zagier normalized the parameter $a$ to be
  positive. With $(9,3,27)$ instead of $(-9,-3,27)$, the modular
  function $t$ and the modular form $f$ are
  $$
  \frac{\eta(\tau)^3\eta(4\tau)^3\eta(18\tau)^9}
  {\eta(2\tau)^9\eta(9\tau)^3\eta(36\tau)^3}
  =-\frac{\eta(9(\tau+1/2))^3}{\eta(\tau+1/2)^3}
  $$
  and
  $$
  \frac{\eta(2\tau)^9\eta(3\tau)\eta(12\tau)}
  {\eta(\tau)^3\eta(4\tau)^3\eta(6\tau)^3}
  =\frac{\eta(\tau+1/2)^3}{\eta(3(\tau+1/2))},
  $$
  respectively. In other words, the congruence group with this choice
  of $(a,b,c)$ is
  $$
  \M1{1/2}01\Gamma_0(9)\M1{1/2}01^{-1}.
  $$
  Likewise, the case $(-17,-6,72)$ is listed as $(17,6,72)$ with the
  modular function $t$ and $f$ given by
  $$
  \frac{\eta(\tau)^5\eta(3\tau)\eta(4\tau)^5\eta(6\tau)^2\eta(12\tau)}
  {\eta(2\tau)^{14}}=-\frac{\eta(2(\tau+1/2))\eta(6(\tau+1/2))^5}
  {\eta(\tau+1/2)^5\eta(3(\tau+1/2))}
  $$
  and
  $$
  \frac{\eta(2\tau)^{15}\eta(3\tau)^2\eta(12\tau)^2}
  {\eta(\tau)^6\eta(4\tau)^6\eta(6\tau)^5}
  =\frac{\eta(\tau+1/2)^6\eta(6(\tau+1/2))}
  {\eta(2(\tau+1/2)^3\eta(3(\tau+1/2))^2},
  $$
  respectively, in \cite{Wan-Zudilin,Zagier-Apery}.
\end{Remark}

Let $P_n(x)$ be the $n$th
Legendre polynomial defined by \eqref{equation: Legendre}.
In \cite[Theorem 2]{Wan-Zudilin}, Wan and Zudilin proved that
\begin{equation*}
  \begin{split}
    &\sum_{n=0}^\infty u_nP_n\left(
      \frac{(X+Y)(1+cXY)-2aXY}{(Y-X)(1-cXY)}\right)
    \left(\frac{Y-X}{1-cXY}\right)^n \\
    &\qquad=(1-cXY)\left(\sum_{n=0}^\infty u_nX^n\right)
    \left(\sum_{n=0}^\infty u_nY^n\right).
  \end{split}
\end{equation*}
Now it is straightforward to prove by induction that
\begin{equation} \label{equation: combinatorial}
\sum_{m=0}^n\binom nm^2r^ms^{n-m}
=\sum_{m=0}^{\lfloor n/2\rfloor}\binom{2m}m\binom n{2m}
(rs)^m(r+s)^{n-2m}.
\end{equation}
Thus, setting
\begin{equation} \label{equation: Apery xy}
x=\frac{(X+Y)(1+cXY)-2aXY}{(1-cXY)^2}, \quad
y=\frac{XY(1-aX+cX^2)(1-aY+cY^2)}{((X+Y)(1+cXY)-2aXY)^2},
\end{equation}
we may write Wan and Zudilin's formula as
\begin{equation} \label{equation: WZ identity}
  \begin{split}
    \sum_{n=0}^\infty u_n\sum_{m=0}^{\lfloor n/2\rfloor}
    \binom{2m}m\binom n{2m}x^ny^m
    =(1-cXY)\left(\sum_{n=0}^\infty u_nX^n\right)
    \left(\sum_{n=0}^\infty u_nY^n\right).
  \end{split}
\end{equation}
\begin{Remark}
A proof of this formula different from the one given in \cite{Wan-Zudilin} was provided by Chan \cite{Chan}. Chan and Tanigawa \cite{Chan-T} also gave a generalization of this formula.
Just like the role of Clausen's identity in establishing Ramanujan's series for $1/\pi$, the identity \eqref{equation: WZ identity} will be  a key in proving 2-variable $1/\pi$-series in this paper. There are more formulas of similar nature in the literature. See the works of Chan et al. \cite{Chan-Tanigawa} or  Straub and Zudilin  \cite[Theorem 1]{Straub-Zudilin}, for example.
\end{Remark}

We now describe the bimodular meaning of the identity  \eqref{equation: WZ identity} and use Theorem 2.1 of \cite{Yang-Yui} to give a new proof of it (see Corollary \ref{cor-WZ-id}).

Assume that $c\neq 0$ and $a^2-4c\neq 0$. Let $\alpha$ and $\beta$ be
the two numbers such that
$$
1-aX+cX^2=(1-\alpha X)(1-\beta X).
$$
Then $x$ and $y$ are invariant under
$$
(X,Y)\longmapsto\left(\frac1{cX},\frac1{cY}\right), \quad
(X,Y)\longmapsto\left(\frac{1-\alpha X}{\alpha(1-\beta X)},
  \frac{1-\alpha Y}{\alpha(1-\beta Y)}\right)
$$
and their composition
$$
(X,Y)\longmapsto\left(\frac{1-\beta X}{\beta(1-\alpha X)},
  \frac{1-\beta Y}{\beta(1-\alpha Y)}\right),
$$
in addition to the obvious symmetry $(X,Y)\mapsto(Y,X)$. When
$(a,b,c)$ is one of the sporadic cases, these symmetries all come from
normalizers of the corresponding congruence subgroups in
$\SL(2,\R)$.

\begin{Proposition} \label{proposition: sporadic bimodular}
  Let $(a,b,c)$ be one of the six sporadic cases, and $\Gamma$,
  $t(\tau)$ and $f(\tau)$ be given as in Table \ref{table: sporadic}.
  Then the involutions $\sigma_1:t\mapsto1/ct$ and
  $\sigma_2:t\mapsto(1-\alpha t)/\alpha(1-\beta t)$ correspond to the
  action of normalizers of  $\Gamma$ in $\SL(2,\R)$ given by
  $$\extrarowheight3pt
  \begin{array}{lllll} \hline\hline
  (a,b,c) & \alpha & \beta & \sigma_1 & \sigma_2 \\ \hline
  (7,2,-8) &  -1 & 8 & w_2 & w_3 \\
  (10,3,9) &  1 & 9 & w_3  & w_2 \\
  (-17,-6,72) & -8 & -9 & w_6 & w_2 \\
  (12,4,32) & 4 & 8 & \SM4184 & \SM2{-1}8{-2} \\
  (-9,-3,27) & (-9+3\sqrt{-3})/2 & (-9-3\sqrt{-3})/2 & w_9
                                      &\SM{-3}{-2}93\\
  (11,3,-1) & (11+5\sqrt5)/2 & (11-5\sqrt5)/2
            & \SM2{-1}5{-2} & \SM0{-1}50 \\
  \hline\hline
  \end{array}
  $$
  (For convenience, we represent elements in $N(\Gamma)$ by matrices
  in $\GL^+(2,\Q)$, instead of matrices in $\SL(2,\R)$.)
  Moreover, set $t_1=t(\tau_1)$, $t_2=t(\tau_2)$,
$$
x(\tau_1,\tau_2)=\frac{(t_1+t_2)(1+ct_1t_2)-2at_1t_2}{(1-ct_1t_2)^2},
$$
$$
y(\tau_1,\tau_2)=\frac{t_1t_2(1-at_1+ct_1^2)(1-at_2+ct_2^2)}
{((t_1+t_2)(1+ct_1t_2)-2at_1t_2)^2},
$$
and
$$
F(\tau_1,\tau_2)=(1-ct(\tau_1)t(\tau_2))f(\tau_1)f(\tau_2).
$$
Then $x$ and $y$ are bimodular functions on $(\Gamma,G)$, and $F$ is a
bimodular form of weight $1$ on $(\Gamma,G)$ with characters
$(\chi_1,\chi_2)$, where $\Gamma$, $G$, $\chi_1$, and $\chi_2$ are
given by
$$ \extrarowheight3pt \small
\begin{array}{lllll} \hline\hline
  (a,b,c) & \Gamma & G & \chi_1 & \chi_2 \\ \hline
  (7,2,-8) & \Gamma_0(6) & \Gamma_0(6)+w_2,w_3 & \JS{-3}\cdot &
    \chi_2(w_2)=1,\chi_2(w_3)=-1\\
  (10,3,9) & \Gamma_0(6) & \Gamma_0(6)+w_2,w_3 & \JS{-3}\cdot&
    \chi_2(w_2)=-1,\chi_2(w_3)=1\\
  (-17,-6,72) & \Gamma_0(6) & \Gamma_0(6)+w_2,w_3 & \JS{-3}\cdot&
    \chi_2(w_2)=-1,\chi_2(w_3)=-1\\
  (12,4,32) & \Gamma_0(8) & \Gamma_0(8)+\SM4184,w_8 & \JS{-4}\cdot
    &\chi_2\left(\SM4184\right)=1,\chi_2(w_8)=-1 \\
  (-9,-3,27) & \Gamma_0(9) & \Gamma_0(9)+\SM{-3}{-2}9{3},w_9
     & \JS{-3}\cdot& \chi_2\left(\SM{-3}{-2}93\right)=-1,
                                       \chi_2(w_9)=1\\
  (11,3,-1) & \Gamma_1(5) & \Gamma_0(5)+w_5 & \chi
     & \chi_2\left(\SM2{-1}5{-2}\right)=1,
       \chi_2(w_5)=-1\\
  \hline\hline
\end{array}
$$
and $w_m$ denotes the Atkin-Lehner involution. Here the character
$\chi$ for the case $(11,3,-1)$ is
$$
\chi(\gamma)=\begin{cases}
  1, &\text{if }\gamma\equiv\M1\ast01\mod 5,\\
  -1,&\text{if }\gamma\equiv\M{-1}\ast0{-1}\mod 5. \end{cases}
$$
(Note that since $\chi_1$ is a quadratic character in each case, the
value of $\chi_2$ for $w_m$ does not depend on the choice of
representatives for $w_m$.)
\end{Proposition}

Since the modular functions and modular forms are expressed in terms of the Dedekind eta function $\eta(\tau)$, here let us recall the transformation formula for $\eta(\tau)$ first.

\begin{Lemma}[{\cite[Pages 125--127]{Weber}}] \label{lemma: eta}
For
$$
  \gamma=\begin{pmatrix}a&b\\ c&d\end{pmatrix}\in\SL(2,\mathbb Z),
$$
the transformation formula for $\eta(\tau)$ is given by, for $c=0$,
$$
  \eta(\tau+b)=e^{\pi ib/12}\eta(\tau),
$$
and, for $c\neq 0$,
$$
  \eta(\gamma\tau)=\epsilon_1(a,b,c,d)\sqrt{\frac{c\tau+d}i}\eta(\tau)
$$
with
\begin{equation}
\label{epsilon1}
  \epsilon_1(a,b,c,d)=
  \begin{cases}\displaystyle
   \left(\frac dc\right)i^{(1-c)/2}
   e^{\pi i\left(bd(1-c^2)+c(a+d)\right)/12},
    &\text{if }c\text{ is odd},\\
  \displaystyle
  \left(\frac cd\right)e^{\pi i\left(ac(1-d^2)+d(b-c+3)\right)/12},
    &\text{if }d\text{ is odd},
  \end{cases}
\end{equation}
where $\displaystyle\left(\frac dc\right)$ is the Legendre-Jacobi
symbol.
\end{Lemma}

\begin{proof}[Proof of Proposition \ref{proposition: sporadic bimodular}] When $t$ and $f$ are eta-products and the normalizer
  comes from Atkin-Lehner involutions, it is straightforward to use
  the transformation law for the Dedekind eta function (Lemma
  \ref{lemma: eta}) to verify the bimodular properties of $t$ and
  $F$. For instance, when the group is $\Gamma_0(6)$ and the
  Atkin-Lehner involution is $w_2$, represented by the matrix
  $\SM2{-1}6{-2}$, we have
  \begin{equation*}
  \begin{split}
  \eta\left(\M2{-1}6{-2}\tau\right)
  &=\eta\left(\M1{-1}3{-2}(2\tau)\right)
  =e^{-2\pi i/24}\sqrt{\frac{6\tau-2}i}\eta(2\tau), \\
  \eta\left(2\M2{-1}6{-2}\tau\right)
  &=\eta\left(\M2{-1}3{-1}\tau\right)
  =e^{2\pi i/24}\sqrt{\frac{3\tau-1}i}\eta(\tau), \\
  \eta\left(3\M2{-1}6{-2}\tau\right)
  &=\eta\left(\M1{-3}1{-2}(6\tau)\right)
  =e^{-2\pi i/24}\sqrt{\frac{6\tau-2}i}\eta(6\tau), \\
  \eta\left(6\M2{-1}6{-2}\tau\right)
  &=\eta\left(\M2{-3}1{-1}(3\tau)\right)
  =e^{2\pi i/24}\sqrt{\frac{3\tau-1}i}\eta(3\tau).
  \end{split}
  \end{equation*}
  Thus, for the case $(a,b,c)=(7,2,-8)$, we have
  $$
  t(\tau)\big|w_2=-\frac18\frac{\eta(2\tau)^3\eta(3\tau)^9}
  {\eta(\tau)^3\eta(6\tau)^9}=-\frac1{8t},
  $$
  $$
  f(\tau)\big|w_2=\frac{\sqrt2e^{-2\pi i/4}}{6\tau-2}
  \frac{4(3\tau-1)}i\frac{\eta(\tau)\eta(6\tau)^6}
  {\eta(2\tau)^2\eta(3\tau)^3}
  =-\sqrt 8t(\tau)f(\tau),
  $$
  and
  $$
  F(\tau_1,\tau_2)\big|w_2=\left(1+\frac1{8t(\tau_1)t(\tau_2)}
  \right)8t(\tau_1)t(\tau_2)f(\tau_1)f(\tau_2)
  =F(\tau_1,\tau_2),
  $$
  where, for $\gamma=\SM abcd\in\mathrm{GL}^+(2,\Q)$,
  $F(\tau_1,\tau_2)|\gamma$ is defined to be
  $$
  F(\tau_1,\tau_2)\big|\gamma:=
  \frac{\det\gamma}{(c\tau_1+d)(c\tau_2+d)}
  F(\gamma\tau_1,\gamma\tau_2).
  $$
  Here we omit the proof of the remaining cases where $t$ and $f$ are
  eta-products and the normalizer comes from Atkin-Lehner involutions.

  When the normalizer is not from Atkin-Lehner involutions, the
  computation is a little more complicated. In the case
  $\Gamma=\Gamma_0(8)$ and the normalizer is $\gamma=\SM4184$, we have
  \begin{equation*}
    \begin{split}
      \eta\left(\M4184\tau\right)
      &=\eta\left(\M1021(2\tau+1/2)\right)
      =e^{2\pi i/24}\sqrt{\frac{4\tau+2}i}\eta(2\tau+1/2) \\
      &=e^{2\pi i/16}\sqrt{\frac{4\tau+2}i}
      \frac{\eta(4\tau)^3}{\eta(2\tau)\eta(8\tau)}, \\
      \eta\left(2\M4184\tau\right)
      &=\eta\left(\M1112(4\tau)\right)
      =e^{2\pi i/8}\sqrt{\frac{4\tau+2}i}\eta(4\tau), \\
      \eta\left(4\M4184\tau\right)
      &=\eta\left(\M2111(2\tau)\right)
      =e^{2\pi i/8}\sqrt{\frac{2\tau+1}i}\eta(2\tau), \\
      \eta\left(8\M4184\tau\right)
    &=\eta\left(\M4{-1}10(\tau+1/2)\right)
    =e^{2\pi i/6}\sqrt{\frac{2\tau+1}{2i}}\eta(\tau+1/2) \\
    &=e^{6\pi i/16}\sqrt{\frac{2\tau+1}{2i}}
    \frac{\eta(2\tau)^3}{\eta(\tau)\eta(4\tau)}.
    \end{split}
  \end{equation*}
  From these, we easily deduce the transformation formulas for $t$,
  $f$, and $F$ under $\SM4184$.

  For $\Gamma=\Gamma_0(9)$ with normalizer $\SM{-3}{-2}9{3}$, we have
  \begin{equation*}
    \begin{split}
      \eta\left(\M{-3}{-2}9{3}\tau\right)
      &=\eta\left(\M{-1}03{-1}(\tau+2/3)\right)
      =\sqrt{\frac{3\tau+1}i}\eta(\tau+2/3), \\
      \eta\left(3\M{-3}{-2}9{3}\tau\right)
      &=\eta\left(\M{-1}{-2}1{1}(3\tau)\right)
      =\sqrt{\frac{3\tau+1}i}\eta(3\tau), \\
      \eta\left(9\M{-3}{-2}9{3}\tau\right)
      &=\eta\left(\M{-3}{-1}10(\tau+1/3)\right)
      =e^{-2\pi i/8}\sqrt{\frac{\tau+1/3}i}\eta(\tau+1/3).
    \end{split}
  \end{equation*}
  This yields the transformation formulas for $t$, $f$, and $F$ under
  $\SM{-3}{-2}93$.

  The case $(11,3,-1)$ is more complicated. For integers $g$ and $h$
  not congruent to $0$ modulo $5$ simultaneously, define the
  generalized Dedekind eta function $E_{g,h}(\tau)$ by
  $$
  E_{g,h}(\tau):=q^{B(g/5)/2}\prod_{m=1}^\infty
  \left(1-\zeta^hq^{m-1+g/5}\right)
  \left(1-\zeta^{-h}q^{m-g/5}\right), \quad \zeta=e^{2\pi i/5},
  $$
  where $B(x)=x^2-x+1/6$ is the second Bernoulli polynomial. They are
  modular functions on some congruence subgroup of $\SL(2,\Z)$.
  Then the modular function $t(\tau)$ and the modular form $f(\tau)$
  can be written as
  $$
  t(\tau)=\frac{E_{1,0}(5\tau)^5}{E_{2,0}(5\tau)^5}, \qquad
  f(\tau)=\frac{\eta(5\tau)^2E_{2,0}(5\tau)^2}{E_{1,0}(5\tau)^3},
  $$
  respectively. Using transformation laws for generalized Dedekind eta
  functions \cite[Theorem 1]{Yang-Dedekind}, we find that
  \begin{equation*}
    \begin{split}
      E_{1,0}\left(5\M2{-1}5{-2}\tau\right)
      &=E_{1,0}\left(\M2{-5}1{-2}(5\tau)\right)
      =ie^{-2\pi i/5}E_{2,0}(5\tau), \\
      E_{2,0}\left(5\M2{-1}5{-2}\tau\right)
      &=E_{2,0}\left(\M2{-5}1{-2}(5\tau)\right)
      =-ie^{2\pi i/5}E_{1,0}(5\tau), \\
      \eta\left(5\M2{-1}5{-2}\tau\right)
      &=\eta\left(\M2{-5}1{-2}(5\tau)\right)
      =\sqrt{\frac{5\tau-2}i}\eta(5\tau).
    \end{split}
  \end{equation*}
  It follows that
  \begin{equation*}
    \begin{split}
  t(\tau)\Big|\M2{-1}5{-2}&=-\frac{E_{2,0}(5\tau)^5}{E_{1,0}(5\tau)^5}
  =-\frac1{t(\tau)}, \\
  f(\tau)\Big|\M2{-1}5{-2}
  &=-\frac{\eta(5\tau)^2E_{1,0}(5\tau)^2}{E_{2,0}(5\tau)^3}
  =-t(\tau)f(\tau),
    \end{split}
  \end{equation*}
  and
  $$
  F(\tau_1,\tau_2)\Big|\M2{-1}5{-2}=F(\tau_1,\tau_2).
  $$
  Also,
  \begin{equation*}
    \begin{split}
      E_{1,0}\left(5\M0{-1}50\tau\right)
      &=E_{1,0}\left(\M0{-1}10\tau\right)
      =\frac{e^{2\pi i/10}}iE_{0,-1}(\tau), \\
      E_{2,0}\left(5\M0{-1}50\tau\right)
      &=E_{2,0}\left(\M0{-1}10\tau\right)
      =\frac{e^{2\pi i/5}}iE_{0,-2}(\tau), \\
      \eta\left(5\M0{-1}50\tau\right)
      &=\eta\left(\M0{-1}10\tau\right)
      =\sqrt{\frac\tau i}\eta(\tau).
    \end{split}
  \end{equation*}
  Hence,
  \begin{equation*}
    \begin{split}
      t(\tau)\Big|\M0{-1}50&=-\frac{E_{0,-1}(\tau)^5}{E_{0,-2}(\tau)^5}
      =\frac{1-\alpha t(\tau)}{\alpha(1-\beta t(\tau))}, \\
      f(\tau)\Big|\M0{-1}50&=\frac{e^{2\pi i/10}}{\sqrt5}
      \frac{\eta(\tau)^2E_{0,-2}(\tau)^2}{E_{0,-1}(\tau)^3}
      =-i\sqrt{\frac\alpha{5\sqrt5}}(1-\beta t(\tau))f(\tau),
    \end{split}
  \end{equation*}
  and
  $$
  F(\tau_1,\tau_2)\Big|\M0{-1}50=-F(\tau_1,\tau_2),
  $$
  where $\alpha=(11+5\sqrt5)/2$ and $\beta=(11-5\sqrt5)/2$.
\end{proof}

We now determine the partial differential equations satisfied by $x$, $y$, and $F$.

\begin{Proposition} \label{proposition: sporadic PDEs}
  Let $x$, $y$, and $F$ be given as in Proposition \ref{proposition: sporadic bimodular}. Then $F$, as a function of $x$ and $y$, satisfies the system of partial differential equations
  \begin{equation} \label{equation: PDEs}
    \begin{split}
      &\theta_x(\theta_x-2\theta_y)F-x(a\theta_x^2+a\theta_x+b)F
       +cx^2(\theta_x+1)^2F\\
      &\qquad+cx^2(\theta_x+1)(4y\theta_x+(2-8y)\theta_y)F=0, \\
      &\theta_y^2F-y(2\theta_y-\theta_x+1)(2\theta_y-\theta_x)F=0.
    \end{split}
  \end{equation}
\end{Proposition}

\begin{proof}
For $j=1,2$, set $q_j=e^{2\pi i \tau_j}$,
$$
G_{x,j}=\frac{\theta_{q_j}x}x, \quad
G_{y,j}=\frac{\theta_{q_j}y}y, \quad
G_{F,j}=\frac{\theta_{q_j}F}F,
$$
and
$$
G_{t_j}=\frac{\theta_{q_j}t_j}{t_j}, \quad
G_{f_j}=\frac{\theta_{q_j}f_j}{f_j}.
$$
By Theorem 1 of \cite{Yang-ModDiff} or by verifying directly, we have
\begin{equation} \label{equation: Gt Gf}
\frac{\theta_{q_j}G_{t_j}-2G_{t_j}G_{f_j}}{G_{t_j}^2}
=\frac{-at_j+2ct_j^2}{1-at_j+ct_j^2}, \quad
\frac{\theta_{q_j}G_{f_j}-G_{f_j}^2}{G_{t_j}^2}
=\frac{bt_j-ct_j^2}{1-at_j+ct_j^2}
\end{equation}
in all cases.
According to Theorem 2.1 of \cite{Yang-Yui}, $F$, as a function of $x$ and $y$, satisfies
\begin{equation} \label{equation: DE coeffs}
  \begin{split}
  \theta_x^2F+r_0\theta_x\theta_yF+r_1\theta_xF+r_2\theta_yF+r_3F=0,\\
  \theta_y^2F+s_0\theta_x\theta_yF+s_1\theta_xF+s_2\theta_yF+s_3F=0,
\end{split}
\end{equation}
where, using \eqref{equation: Gt Gf} and the chain rule,
\begin{equation*}
  \begin{split}
  r_0&=\frac{2G_{y,1}G_{y,2}}{G_{x,1}G_{y,2}+G_{y,1}G_{x,2}}
  =-\frac{2(1-cx^2+4cx^2y)}{1-ax+cx^2+4cx^2y}, \\
  s_0&=\frac{2G_{x,1}G_{x,2}}{G_{x,1}G_{y,2}+G_{y,1}G_{x,2}}
  =\frac{2y(1-2ax+3cx^2+4cx^2y)}{(1-ax+cx^2+4cx^2y)(1-4y)},
\end{split}
\end{equation*}
  \begin{equation*}
    \begin{split}
    r_1&=\frac{G_{y,2}^2(\theta_{q_1}G_{x,1}-2G_{F,1}G_{x,1})
      -G_{y,1}^2(\theta_{q_2}G_{x,2}-2G_{F,2}G_{x,2})}
    {G_{x,1}^2G_{y,2}^2-G_{y,1}^2G_{x,2}^2} \\
    &=\frac{-ax+2cx^2+4cx^2y}{1-ax+cx^2+4cx^2y}, \\
    s_1&=\frac{-G_{x,2}^2(\theta_{q_1}G_{x,1}-2G_{F,1}G_{x,1})
      +G_{x,1}^2(\theta_{q_2}G_{x,2}-2G_{F,2}G_{x,2})}
    {G_{x,1}^2G_{y,2}^2-G_{y,1}^2G_{x,2}^2} \\
    &=\frac{y(1-2ax+3cx^2+8cx^2y)}{(1-ax+cx^2+4cx^2y)(1-4y)},
  \end{split}
  \end{equation*}
  \begin{equation*}
    \begin{split}
    r_2&=\frac{G_{y,2}^2(\theta_{q_1}G_{y,1}-2G_{F,1}G_{y,1})
      -G_{y,1}^2(\theta_{q_2}G_{y,2}-2G_{F,2}G_{y,2})}
    {G_{x,1}^2G_{y,2}^2-G_{y,1}^2G_{x,2}^2}\\
    &=\frac{2cx^2(1-4y)}{1-ax+cx^2+4cx^2y}, \\
    s_2&=\frac{-G_{x,2}^2(\theta_{q_1}G_{y,1}-2G_{F,1}G_{y,1})
      +G_{x,1}^2(\theta_{q_2}G_{y,2}-2G_{F,2}G_{y,2})}
    {G_{x,1}^2G_{y,2}^2-G_{y,1}^2G_{x,2}^2} \\
    &=\frac{2y(-1+ax-8cx^2y)}{(1-ax+cx^2+4cx^2y)(1-4y)},
  \end{split}
  \end{equation*}
  \begin{equation*}
    \begin{split}
    r_3&=-\frac{G_{y,2}^2(\theta_{q_1}G_{F,1}-G_{F,1}^2)
      -G_{y,1}^2(\theta_{q_2}G_{F,2}-G_{F,2}^2)}
      {G_{x,1}^2G_{y,2}^2-G_{y,1}^2G_{x,2}^2} \\
    &=\frac{-bx+cx^2}{1-ax+cx^2+4cx^2y}, \\
    s_3&=-\frac{-G_{x,2}^2(\theta_{q_1}G_{F,1}-G_{F,1}^2)
      +G_{x,1}^2(\theta_{q_2}G_{F,2}-G_{F,2}^2)}
    {G_{x,1}^2G_{y,2}^2-G_{y,1}^2G_{x,2}^2} \\
    &=\frac{y(-bx+cx^2)}{(1-ax+cx^2+4cx^2y)(1-4y)}.
    \end{split}
  \end{equation*}
Clearing the denominator of the first equation in \eqref{equation: DE coeffs}, we get the first equation in \eqref{equation: PDEs}. Multiplying the two equations in \eqref{equation: DE coeffs} by $-y$ and $1-4y$, respectively, and adding them together, we get the second equation in \eqref{equation: PDEs}.
\end{proof}

Using this proposition, one can give a new proof of Wan and Zudilin's identity \eqref{equation: WZ identity}.

\begin{Corollary}\label{cor-WZ-id}
The identity \eqref{equation: WZ identity} holds for all sporadic Ap\'ery-like sequences $\{u_n\}$.
\end{Corollary}

\begin{proof} We check that the system of partial differential equations in \eqref{equation: PDEs} has only one solution holomorphic near $(0,0)$, up to scalars, and that the left-hand side of \eqref{equation: WZ identity} is a solution with this property.
  In view of Proposition \ref{proposition: sporadic PDEs}, this implies that \eqref{equation: WZ identity} holds for all sporadic Ap\'ery-like sequence $\{u_n\}$.
\end{proof}

Using bimodular forms and arithmetic properties of CM-points, we produce many new two-variable $1/\pi$-series with rational arguments. As mentioned in the introduction section, most of the series are coming from imaginary quadratic orders with class groups being isomorphic to an elementary $2$-group of order $8$.

\begin{Theorem} Let $\{u_n\}$ be one of the six sporadic Ap\'ery-like
  sequences with parameters $(a,b,c)$. Then we have
\begin{align}\label{series-form}
  \sum_{n=0}^\infty\sum_{m=0}^{\lfloor n/2\rfloor}u_n
  \binom n{2m}\binom{2m}m(An+B)x^ny^m=\frac C\pi
\end{align}
  for $A$, $B$, $C$, $x$, and $y$ given in Tables
  \ref{table: 6-1}--\ref{table: 8} in Appendix \ref{section: sporadic
    data}.
\end{Theorem}

The proof will be given in Section \ref{section: pi series}.

\begin{Remark}
Using the same techniques, we also obtain many $2$-variable
$1/\pi$-series of the form \eqref{series-form} for
$u_n=(a)_n(1-a)_n/(n!)^2$, $a\in\{1/2,1/3,1/4\}$. They are given in
Tables \ref{table:a=2}--\ref{table:a=4} in Appendix
\ref{section:hyper-data}. They include all series conjectured or
proved in \cite{CWZ,Sun-list}, along with many new series missed by
\cite{CWZ,Sun-list}.
\end{Remark}

Note that for each $1/\pi$-series in the tables, there is a companion
series of the form
\begin{align}\label{spo-companion}
\sum_{n=0}^\infty\sum_{m=0}^{\lfloor n/2\rfloor}u_n
\binom{2m}m\binom n{2m}(A'm+B')x^ny^m=\frac{C'}\pi
\end{align}
(see Theorem \ref{theorem: general}). For example, for the case $(a,b,c)=(7,2,-8)$, there are series of the form \eqref{spo-companion} for
\begin{align*}
(x,y,A',B',C')=(47/441,1/47^2, 2835, 172, 402\sqrt{5}), \\
 (-97/1176, 1/994^2, 1164240, 43269, 53627\sqrt{5}), \,\, \text{etc.}
\end{align*}
Similarly, for $u_n=(1/2)_n^2/(n!)^2$, there are series of the form
\eqref{spo-companion} for
$$(x,y,A',B',C')=(-1/16, 16, 105, 12, 44), \quad (-17/32, 1/34^2, 240, 11, 31) \,\, \text{etc.}$$
However, in general the constants $A'$, $B'$, and $C'$ are much more
complicated (see the discussion following Theorem \ref{theorem:
  general}), so we will not list them in the paper.

\section{Bimodular properties of other known two-variable $1/\pi$-series}

In Sections 2 and 3 of \cite{Sun-list}, Sun listed several families of
conjectural $1/\pi$-series of the form
$$
\sum_{n=0}^\infty(An+B)\sum_{m=0}^n a_{n,m}x^ny^m=\frac C\pi
$$
for some coefficients $a_{n,m}$ expressible in terms of binomial
coefficients, some rational numbers $A$, $B$, $x$, $y$, and some
algebraic number $C$ (Conjectures 2, 3(ii), 3(iii), 6(i), 6(ii), and
I--VII). A close inspection suggests that $x$ and $y$ in several
families, including those in Conjectures 2, 6(i), and 6(ii), are
related by algebraic functions. In addition, Zudilin \cite{Zudilin-T2}
found that $x$ and $y$ in Conjecture VII are parameterized by a single
modular function on $\Gamma_0(7)$. Those series should be regarded as
one-variable $1/\pi$-series and will not be discussed here. The
bimodular properties of the remaining case are described below.
\medskip

\paragraph{\bf Series in Conjectures I-III} The series in Sun's Conjectures
I-III are of the form
$$
\sum_{n=0}^\infty\frac{(a)_n(1-a)_n}{(n!)^2}T_n(b,c)z^n
=\sum_{n=0}^\infty\frac{(a)_n(1-a)_n}{(n!)^2}T_n(bz,cz^2)
$$
with $a=1/2$,
$1/3$, and $1/4$, respectively. As mentioned in the introduction
section, they can also be written as
$$
\sum_{n=0}^\infty\sum_{m=0}^{\lfloor n/2\rfloor}
\frac{(a)_n(1-a)_n}{(n!)^2}\binom n{2m}\binom{2m}mx^ny^m,
\quad
x=bz,~y=\frac c{b^2}.
$$
The bimodular properties of these series were already studied in
\cite[Theorem 4.1]{Yang-Yui}.

\begin{Proposition}\label{prop-hyper}
  For $a=1/2$, let
  $$
  t(\tau)=\frac{\theta_2(\tau)^4}{\theta_3(\tau)^4}, \qquad
  f(\tau)=\theta_4(\tau)^2,
  $$
  where
  $$
  \theta_2(\tau)=\sum_{n\in\Z}q^{(n+1/2)^2}, \quad
  \theta_3(\tau)=\sum_{n\in\Z}q^{n^2}, \quad
  \theta_4(\tau)=\sum_{n\in\Z}(-1)^nq^{n^2}
  $$
  are the Jacobi theta functions.
  For $a=1/3$, let
  $$
  t(\tau)=-27\frac{\eta(3\tau)^{12}}{\eta(\tau)^{12}}, \qquad
  f(\tau)=\sum_{m,n\in\Z}q^{m^2+mn+n^2}.
  $$
  For $a=1/4$, let
  $$
  t(\tau)=-64\frac{\eta(2\tau)^{24}}{\eta(\tau)^{24}}, \qquad
  f(\tau)=(2E_2(2\tau)-E_2(\tau))^{1/2},
  $$
  where $E_2(\tau)=1-24\sum_{n=1}^\infty nq^n/(1-q^n)$ is the
  Eisenstein series of weight $2$ on $\SL(2,\Z)$.
  Set $t_1=t(\tau_1)$, $t_2=t(\tau_2)$,
  $$
  x(\tau_1,\tau_2)=-\frac{t_1+t_2}{(1-t_1)(1-t_2)}, \quad
  y(\tau_1,\tau_2)=\frac{t_1t_2}{(t_1+t_2)^2}, \quad
  F(\tau_1,\tau_2)=f(\tau_1)f(\tau_2).
  $$
  Then we have
  \begin{equation} \label{equation: hypergeometric cases}
  F(\tau_1,\tau_2)=\sum_{n=0}^\infty\sum_{m=0}^{\lfloor n/2\rfloor}
  \frac{(a)_n(1-a)_n}{(n!)^2}\binom n{2m}\binom{2m}mx^ny^m,
  \end{equation}
  and it satisfies the system of partial differential equations in
  \eqref{equation: hypergeometric PDE}.

  Moreover, for $a=1/2$, $x$ and $y$ are bimodular functions and $F$
  is a bimodular form of weight $1$ on
  $(\Gamma_0(4),\Gamma_0(4)+w)$, $w=\SM1021$, with characters
  $$
  \chi_1\left(\M\ast\ast cd\right)=(-1)^{c/4}\JS{-4}d, \qquad
  \chi_2(w)=-1.
  $$
  For $a=1/3$, $x$ and $y$ are bimodular functions and $F$ is a
  bimodular form of weight $1$ on $(\Gamma_0(3),\Gamma_0(3)+w_3)$ with
  characters
  $$
  \chi_1\left(\M\ast\ast\ast d\right)=\JS{-3}d, \qquad
  \chi_2(w_3)=-1.
  $$
  For $a=1/4$, $x$ and $y$ are bimodular functions and $F^2$ is a
  bimodular form of weight $2$ on $(\Gamma_0(2),\Gamma_0(2)+w_2)$ with
  trivial characters.
\end{Proposition}

\begin{proof} The identity \eqref{equation: hypergeometric cases} was
  proved in \cite{Yang-Yui}. It is clear from the criterion given in
  \cite[Proposition 3.2.8]{Ligozat} that the functions $t(\tau)$ are
  modular on their respective groups. Moreover, for the cases $a=1/3$
  and $a=1/4$, it is easy to see that their respective Atkin-Lehner
  involutions map $t$ to $1/t$ and hence $x$ and $y$ are bimodular
  functions on the given groups. For $a=1/2$, we note that
  $$
  \theta_2(\tau)=2\frac{\eta(4\tau)^2}{\eta(2\tau)}, \qquad
  \theta_3(\tau)=\frac{\eta(2\tau)^5}{\eta(\tau)^2\eta(4\tau)^2},
  \qquad
  \theta_4(\tau)=\frac{\eta(\tau)^2}{\eta(2\tau)},
  $$
  and
  $$
  t(\tau)=16\frac{\eta(\tau)^8\eta(4\tau)^{16}}{\eta(2\tau)^{24}}.
  $$
  We then compute that
  \begin{equation*}
    \begin{split}
      \eta\left(\M1021\tau\right)&=e^{2\pi i/24}
      \sqrt{\frac{2\tau+1}i}\eta(\tau), \\
      \eta\left(2\M1021\tau\right)
      &=\eta\left(\M1011(2\tau)\right)
      =e^{2\pi i/12}\sqrt{\frac{2\tau+1}i}\eta(2\tau), \\
      \eta\left(4\M1021\tau\right)
      &=\eta\left(\M2{-1}10(\tau+1/2)\right)
      =e^{2\pi i/12}\sqrt{\frac{2\tau+1}{2i}}\eta(\tau+1/2) \\
      &=e^{10\pi i/48}\sqrt{\frac{2\tau+1}{2i}}
        \frac{\eta(2\tau)^3}{\eta(\tau)\eta(4\tau)}
    \end{split}
  \end{equation*}
  and hence
  $$
  t\left(\M1021\tau\right)
  =\frac1{t(\tau)}.
  $$
  It follows that $x$ and $y$ are bimodular functions on
  $(\Gamma_0(4),\Gamma_0(4)+w)$, $w=\SM1021$ in the case $a=1/2$.

  The bimodular property of $F^2$ in the case $a=1/4$ is obvious.
  That of $F$ in the case $a=1/3$ follows from general properties of
  the theta series $\sum_{m,n}q^{m^2+mn+n^2}$. Finally, for the case
  $a=1/2$, the computation above shows that
  $$
  f(\tau)\Big|\M1021=\frac1if(\tau), \qquad
  F(\tau_1,\tau_2)\Big|\M1021=-F(\tau_1,\tau_2).
  $$
  This completes the proof.
\end{proof}

\paragraph{\bf Series in Conjectures IV and 3(ii)(iii)}

In Conjecture IV of \cite{Sun-list}, Sun listed a family of
$1/\pi$-series of the form
$$
\sum_{n=0}^\infty\binom{2n}n^2T_{2n}(b,c)(An+B)C^n=\frac D\pi.
$$
These formulas were proved in \cite{Wan-Zudilin}. The key ingredient
in the proof is Wan and Zudilin's identity (\cite[Theorem
3]{Wan-Zudilin})
\begin{equation} \label{equation: WZ 1}
  \begin{split}
   &\sum_{n=0}^\infty\frac{(1/2)_n^2}{(n!)^2}
   P_{2n}\left(\frac{(X+Y)(1-XY)}{(X-Y)(1+XY)}\right)
   \left(\frac{X-Y}{1+XY}\right)^{2n} \\
   &\qquad=\frac{1+XY}2{}_2F_1(1/2,1/2;1;1-X^2)
   {}_2F_1(1/2,1/2;1;1-Y^2),
  \end{split}
\end{equation}
valid for $(X,Y)$ close to $(1,1)$, where $P_n(x)$ is the $n$th
Legendre polynomial defined by \eqref{equation: Legendre}. Using
\eqref{equation: combinatorial}, we see that
\begin{equation*}
  \begin{split}
   &P_n\left(\frac{(X+Y)(1-XY)}{(X-Y)(1+XY)}\right) \\
   &\qquad
   =\sum_{m=0}^n\binom nm^2\left(\frac{Y(1-X^2)}{(X-Y)(1+XY)}\right)^m
   \left(\frac{X(1-Y^2)}{(X-Y)(1+XY)}\right)^{n-m} \\
   &\qquad
   =\frac1{(X-Y)^n(1+XY)^n}T_n\left((X+Y)(1-XY), XY(1-X^2)(1-Y^2)
     \right).
  \end{split}
\end{equation*}
Thus, \eqref{equation: WZ 1} can be written as
\begin{equation*}
  \begin{split}
   &\frac{1+XY}2{}_2F_1(1/2,1/2;1;1-X^2)
   {}_2F_1(1/2,1/2;1;1-Y^2) \\
   &\qquad
   =\sum_{n=0}^\infty\frac{(1/2)_n^2}{(n!)^2}
   \frac{T_{2n}\left((X+Y)(1-XY), XY(1-X^2)(1-Y^2)\right)}
   {(1+XY)^{4n}}.
  \end{split}
\end{equation*}
Letting
$$
x=\frac{(X+Y)^2(1-XY)^2}{(1+XY)^4}, \quad
y=\frac{XY(1-X^2)(1-Y^2)}{(X+Y)^2(1-XY)^2},
$$
we then have
\begin{equation} \label{equation: WZ 2}
  \begin{split}
   &\frac{1+XY}2{}_2F_1(1/2,1/2;1;1-X^2)
   {}_2F_1(1/2,1/2;1;1-Y^2) \\
   &\qquad
   =\sum_{n=0}^\infty\frac{(1/2)_n^2}{(n!)^2}
   \sum_{m=0}^{n}\binom{2m}m\binom{2n}{2m}x^ny^m.
  \end{split}
\end{equation}
Note that $x$ and $y$ are both invariant under
\begin{equation} \label{equation: involutions T2n}
(X,Y)\longmapsto\left(\frac1X,\frac1Y\right), \quad
(X,Y)\longmapsto\left(\frac{X-1}{X+1},\frac{Y-1}{Y+1}\right),
\end{equation}
in addition to the obvious symmetry $(X,Y)\mapsto(Y,X)$. The modular
meanings of these symmetries are as follows.

Recall that
$$
\theta_3(\tau)^2={}_2F_1\left(\frac12,\frac12;1;
  \frac{\theta_2(\tau)^4}{\theta_3(\tau)^4}\right)
$$
and
$$
\theta_2(\tau)^4+\theta_4(\tau)^4=\theta_3(\tau)^4.
$$
Therefore, if we let
$$
t(\tau)=\frac{\theta_4(\tau)^2}{\theta_3(\tau)^2}
=\frac{\eta(\tau)^8\eta(4\tau)^4}{\eta(2\tau)^{12}},
$$
$t_1=t(\tau_1)$, $t_2=t(\tau_2)$, and
\begin{equation} \label{equation: T2n xy}
x=\frac{(t_1+t_2)^2(1-t_1t_2)^2}{(1+t_1t_2)^4}, \quad
y=\frac{t_1t_2(1-t_1^2)(1-t_2^2)}{(t_1+t_2)^2(1-t_1t_2)^2},
\end{equation}
then \eqref{equation: WZ 2} just says that
\begin{equation} \label{equation: T2n identity}
\frac{1+t_1t_2}2\theta_3(\tau_1)^2\theta_3(\tau_2)^2
=\sum_{n=0}^\infty\frac{(1/2)_n^2}{(n!)^2}
\sum_{m=0}^n\binom{2m}m\binom{2n}{2m}x^ny^m.
\end{equation}

The function $t(\tau)$ is a modular function on $\Gamma_0(8)$
holomorphic throughout the upper half-plane. Its
values at the four cusps $\infty$, $1/4$, $1/2$, and $0$, of $X_0(8)$
are $1$, $-1$, $\infty$, and $0$, respectively. In particular, it
generates the field of modular functions on $X_0(8)$. Let
$N(\Gamma_0(8))$ be the normalizer of $\Gamma_0(8)$ in
$\SL(2,\R)$. From the
description of $N(\Gamma_0(8))$ given in \cite{Conway-Norton}, we know
that $N(\Gamma_0(8))/\Gamma_0(8)$ is isomorphic to the dihedral group
$D_8$ of $8$ elements and is generated by
$$
\sigma_1=\M0{-1}8{-4}, \qquad\sigma_2=\M0{-1}80,
$$
with $\sigma_1^4=\sigma_2^2=\mathrm{id}$ and
$\sigma_2\sigma_1\sigma_2^{-1}=\sigma_1^{-1}$.
From the values of $t$ at the four cusps, it is easy to see that
$$
t\big|\sigma_1=\frac{t-1}{t+1}, \qquad
t\big|\sigma_2=\frac{1-t}{1+t}.
$$
Therefore, the invariances in \eqref{equation: involutions T2n} mean
that $x(\tau_1,\tau_2)$ and $y(\tau_1,\tau_2)$
are bimodular functions on $(\Gamma_0(8),N(\Gamma_0(8)))$.

Now let $F(\tau_1,\tau_2)$ be the function on the left-hand side of
\eqref{equation: T2n identity}. We have
\begin{equation*}
  \begin{split}
    \eta\left(\M0{-1}8{-4}\tau\right)
    &=\sqrt{\frac{8\tau-4}i}\eta(8\tau-4)
    =e^{-2\pi i/6}\sqrt{\frac{8\tau-4}i}\eta(8\tau), \\
    \eta\left(2\M0{-1}8{-4}\tau\right)
    &=\sqrt{\frac{4\tau-2}i}\eta(4\tau-2)
    =e^{-2\pi i/12}\sqrt{\frac{4\tau-2}i}\eta(4\tau), \\
    \eta\left(4\M0{-1}8{-4}\tau\right)
    &=\sqrt{\frac{2\tau-1}i}\eta(2\tau-1)
    =e^{-2\pi i/24}\sqrt{\frac{2\tau-1}i}\eta(2\tau).
  \end{split}
\end{equation*}
It follows that
$$
\theta_3(\tau)^2\Big|\M0{-1}8{-4}
=\frac{\sqrt 2}i\frac{\eta(4\tau)^{10}}{\eta(2\tau)^4\eta(8\tau)^4}
=\frac1{\sqrt 2i}(1+t(\tau))\theta_3(\tau)^2,
$$
and
\begin{equation*}
  \begin{split}
F(\tau_1,\tau_2)\Big|\M0{-1}8{-4}
&=-\frac14\left(1+\frac{t_1-1}{t_1+1}\frac{t_2-1}{t_2+1}\right)
(1+t_1)(1+t_2)\theta_3(\tau_1)^2\theta_3(\tau_2)^2 \\
&=-F(\tau_1,\tau_2).
\end{split}
\end{equation*}
A similar computation shows that
$F(\tau_1,\tau_2)\big|\SM0{-1}80=-F(\tau_1,\tau_2)$ as well.
We summarize the bimodular properties of the series in Sun's
Conjecture IV in the following proposition.

\begin{Proposition} \label{proposition: T2n}
  Let $x(\tau_1,\tau_2)$ and $y(\tau_1,\tau_2)$ be
  defined by \eqref{equation: T2n xy}. Let also $F(\tau_1,\tau_2)$ be
  the function on the left-hand side of \eqref{equation: T2n
    identity}. Then $x$ and $y$ are bimodular functions and $F$ is a
  bimodular form of weight $1$ on
  $(\Gamma_0(8),N(\Gamma_0(8)))$ with characters $(\chi_1,\chi_2)$
  given by
  $$
  \chi_1\left(\M\ast\ast\ast d\right)=\JS{-1}d, \quad
  \chi_2\left(\M0{-1}8{-4}\right)=\chi_2\left(\M0{-1}80\right)=-1.
  $$
  Moreover, $F$ satisfies the system
  \begin{equation*}
    \begin{split}
      &4\theta_x(\theta_x-\theta_y)F
        -x\left(2\theta_x+1\right)^2F
       -2x\left(2\theta_x+1\right)
        (4y\theta_x+(1-4y)\theta_y)F=0, \\
      &\theta_y^2F-y(2\theta_y-2\theta_x)(2\theta_y-2\theta_x+1)F=0
    \end{split}
  \end{equation*}
of partial differential equations.
\end{Proposition}


We now consider Sun's Conjectures 3(ii) and 3(iii), where the series
are of the form
$$
\sum_{n=0}^\infty(An+B)\binom{2n}n\sum_{m=0}^n\binom nm^2\binom{2m}n
C^nD^{2m-n}=\frac E\pi
$$
and
$$
\sum_{n=0}^\infty(An+B)\binom{2n}n\sum_{m=0}^n\binom nm^2\binom{2m}n
(-1)^mC^nD^{2m-n}=\frac E\pi,
$$
respectively. The bimodular properties of the two types of series are
clearly the same. In \cite[Theorem 2.3]{Rogers-Straub}, Rogers and
Straub show that if $s$ and $t$ are related to $X$ and $Y$ by
\begin{equation} \label{equation: RS}
-st=\left(\frac{X-Y}{4(1+XY)}\right)^2, \qquad
1+\frac{4s}t=\left(\frac{(X+Y)(1-XY)}{(X-Y)(1+XY)}\right)^2,
\end{equation}
then
\begin{equation*}
  \begin{split}
    &\sum_{n=0}^\infty\binom{2n}n\sum_{m=0}^n\binom nm^2\binom{2m}n
    (-1)^ms^nt^{2m-n} \\
    &\qquad\qquad
    =\frac{1+XY}2{}_2F_1(1/2,1/2;1;1-X^2){}_2F_1(1/2,1/2;1;1-Y^2).
  \end{split}
\end{equation*}
Thus, setting
$$
x=\frac st=\frac{XY(1-X^2)(1-Y^2)}{(X-Y)^2(1+XY)^2}, \quad
y=-t^2=\frac{(X-Y)^4}{16XY(1-X^2)(1-Y^2)},
$$
we have
\begin{equation*}
  \begin{split}
    &\sum_{n=0}^\infty\binom{2n}n\sum_{m=0}^n\binom nm^2
    \binom{2m}nx^ny^m \\
    &\qquad\qquad
    =\frac{1+XY}2{}_2F_1(1/2,1/2;1;1-X^2){}_2F_1(1/2,1/2;1;1-Y^2).
  \end{split}
\end{equation*}
As pointed out by Rogers and Straub \cite{Rogers-Straub}, both
expressions in \eqref{equation: RS} are invariant under
\eqref{equation: involutions T2n}.
Therefore, the bimodular properties of this series is the same as
those of the series in Conjecture IV.

\begin{Proposition} Let
  $$
  t(\tau)=\frac{\theta_4(\tau)^2}{\theta_3(\tau)^2}
  =\frac{\eta(\tau)^8\eta(4\tau)^4}{\eta(2\tau)^{12}},
  $$
  $t_1=t(\tau_1)$, $t_2=t(\tau_2)$, and
  $$
  x=\frac{t_1t_2(1-t_1^2)(1-t_2^2)}{(t_1-t_2)^2(1+t_1t_2)^2}, \quad
  y=\frac{(t_1-t_2)^4}{16t_1t_2(1-t_1^2)(1-t_2^2)}.
  $$
  Then
  \begin{equation*}
    \begin{split}
      \frac{1+t_1t_2}2\theta_3(\tau_1)^2\theta_3(\tau_2)^2
     =\sum_{n=0}^\infty\binom{2n}n\sum_{m=0}^n\binom nm^2
      \binom{2m}nx^ny^m. \\
    \end{split}
  \end{equation*}
  Their bimodular properties are the same as those in Proposition
  \ref{proposition: T2n}. They satisfy the system
  \begin{equation*}
    \begin{split}
      &(\theta_x-\theta_y)^2F
      +2x(\theta_x-2\theta_y)(2\theta_x+1)F=0,\\
      &\theta_y(2\theta_y-\theta_x)F-4xy(2\theta_x+1)(2\theta_y+1)F=0
    \end{split}
  \end{equation*}
  of partial differential equations.
\end{Proposition}

\subsection{Series in Conjecture V}

In \cite[Conjecture V]{Sun-list}, Sun recorded one single
$1/\pi$-series
$$
\sum_{n=0}^\infty\frac{(1/3)_n(2/3)_n}{(n!)^2}
T_{3n}(61,1)\frac{1638n+277}{(-80)^{3n}}=\frac{44\sqrt{105}}\pi
$$
involving $T_{3n}$. In \cite[(34)]{Wan-Zudilin}, Wan and Zudilin gave
another example of such series. Here we shall discuss the
bimodular properties of these series.

According to \cite[Theorem 3]{Wan-Zudilin}, one has
\begin{equation*}
  \begin{split}
   &\sum_{n=0}^\infty\frac{(1/3)_n(2/3)_n}{(n!)^2}P_{3n}
    \left(\frac{X+Y-2X^2Y^2}{W(X-Y)}\right)\left(
      \frac{X-Y}W\right)^{3n} \\
    &\qquad=\frac W3{}_2F_1(1/3,2/3;1;1-X^3){}_2F_1(1/3,2/3;1;1-Y^3),
  \end{split}
\end{equation*}
where $W=\sqrt{1+4XY(X+Y)}$, for $(X,Y)$ close to $(1,1)$. Using
\eqref{equation: Legendre} and \eqref{equation: combinatorial}, this
can be written as
\begin{equation*}
  \begin{split}
    &\frac W3{}_2F_1(1/3,2/3;1;1-X^3){}_2F_1(1/3,2/3;1;1-Y^3) \\
    &\qquad=\sum_{n=0}^\infty\frac{(1/3)_n(2/3)_n}{W^{6n}(n!)^2}
     T_{3n}(X+Y-2X^2Y^2,XY(1-X^3)(1-Y^3)).
  \end{split}
\end{equation*}
Setting
$$
x=\frac{(X+Y-2X^2Y^2)^3}{(1+4XY(X+Y))^3}, \quad
y=\frac{XY(1-X^3)(1-Y^3)}{(X+Y-2X^2Y^2)^2},
$$
we find that
\begin{equation*}
  \begin{split}
    &\frac W3{}_2F_1(1/3,2/3;1;1-X^3){}_2F_1(1/3,2/3;1;1-Y^3) \\
    &\qquad=\sum_{n=0}^\infty\frac{(1/3)_n(2/3)_n}{(n!)^2}
    \sum_{m=0}^{\lfloor 3n/2\rfloor}\binom{2m}m\binom{3n}{2m}
    x^ny^m.
  \end{split}
\end{equation*}

On the other hand, Borwein and Borwein \cite[Theorem
2.3]{Borwein-cubic} showed if we set
$$
L(\tau)=\sum_{m,n\in\Z}q^{m^2+mn+n^2}
$$
and
$$
a(\tau):=L(\tau), \quad
b(\tau):=\frac12(3L(3\tau)-L(\tau)), \quad
c(\tau):=\frac12(L(\tau/3)-L(\tau)),
$$
then
$$
a(\tau)={}_2F_1\left(\frac13,\frac23;1;\frac{c(\tau)^3}{a(\tau)^3}\right)
$$
and
$$
a(\tau)^3=b(\tau)^3+c(\tau)^3.
$$
Therefore, letting
$$
t(\tau)=\frac{b(\tau)}{a(\tau)},
$$
$t_1=t(\tau_1)$, $t_2=t(\tau_2)$, and
\begin{equation} \label{equation: xy T3n}
x=\left(\frac{t_1+t_2-2t_1^2t_2^2}{1+4t_1t_2(t_1+t_2)}\right)^3, \qquad
y=\frac{t_1t_2(1-t_1^3)(1-t_2^3)}{(t_1+t_2-2t_1^2t_2^2)^2},
\end{equation}
we have
\begin{equation} \label{equation: cubic}
  \begin{split}
   &\sum_{n=0}^\infty\frac{(1/3)_n(2/3)_n}{(n!)^2}
    \sum_{m=0}^{\lfloor 3n/2\rfloor}\binom{2m}m\binom{3n}{2m}
    x^ny^m \\
  &\qquad=\frac13\sqrt{1+4t_1t_2(t_1+t_2)}{}_2F_1(1/3,2/3;1;1-t_1^3)
  {}_2F_1(1/3,2/3;1;1-t_2^3) \\
  &\qquad=\frac13\sqrt{1+4t_1t_2(t_1+t_2)}a(\tau_1)a(\tau_2).
\end{split}
\end{equation}
The function $t(\tau)$ is modular on $\Gamma_0(9)$. The relation
between $t$ and the Hauptmodul $j_9(\tau)=\eta(\tau)^3/\eta(9\tau)^3$
on $X_0(9)$ is
$$
t=\frac{j_9}{j_9+9}.
$$

Note that
$$
\Gamma_0(9)=\M1003\Gamma(3)\M1003^{-1}.
$$
Since $\Gamma(3)$ is a normal subgroup of $\SL(2,\Z)$,
$\SM1003\SL(2,\Z)\SM1003^{-1}$ normalizes $\Gamma_0(9)$. From
\cite[Section 3]{Conway-Norton}, we see that this is precisely the
normalizer $N(\Gamma_0(9))$ of $\Gamma_0(9)$ in
$\mathrm{PSL}(2,\R)$. The quotient group $N(\Gamma_0(9))/\Gamma_0(9)$
is generated by
$$
\sigma_1=\M0{-1/3}30, \qquad \sigma_2=\M1031,
$$
of orders $2$ and $3$, respectively. Now, we have
$$
t\Big|\sigma_1=\frac{j_9}{j_9+9}\Big|\sigma_1
=\frac{27/j_9}{27/j_9+9}=\frac{1-t}{1+2t}.
$$
For $\sigma_2$, we compute that
\begin{equation*}
  \begin{split}
    \eta\left(\M1031\tau\right)&=\sqrt{\frac{3\tau+1}i}\eta(\tau), \\
    \eta\left(9\M1031\tau\right)&=\eta\left(\M3{-1}10(\tau+1/3)
    \right)=e^{2\pi i/8}\sqrt{\frac{\tau+1/3}i}
    \eta(\tau+1/3),
  \end{split}
\end{equation*}
and
$$
j_9\Big|\M1031=\sqrt{27}e^{-6\pi i/8}
\frac{\eta(\tau)^3}{\eta(\tau+1/3)^3}
=\frac{\beta j_9}{j_9-\alpha},
$$
where $\alpha=(-9+3\sqrt{-3})/2$ and $\beta=(-9-3\sqrt{-3})/2$.
It follows that
$$
t\big|\sigma_2=\rho^2t, \quad \rho=e^{2\pi i/3}.
$$
We check that the functions $x$ and $y$ of $t_1$ and $t_2$ are
invariant under
$$
(t_1,t_2)\mapsto\left(\frac{1-t_1}{1+2t_1},\frac{1-t_2}{1+2t_2}\right),
\quad
(t_1,t_2)\mapsto\left(\rho t_1,\rho t_2\right).
$$
In other words, $x$ and $y$ are bimodular functions on
$(\Gamma_0(9),N(\Gamma_0(9)))$.

Let $F(\tau_1,\tau_2)$ be the function in the last expression of
\eqref{equation: cubic}. We now consider the bimodular property of
$F^2$. We check that
$$
a(\tau)^2=\frac12(3E_2(3\tau)-E_2(\tau)), \qquad
E_2(\tau)=1-24\sum_{n=1}^\infty\frac{nq^n}{1-q^n}.
$$
Then
$$
a(\tau)^2\Big|\sigma_1=\frac12(3E_2(3\tau)-9E_2(9\tau))
=-a(\tau)^2(1+2t(\tau))^2.
$$
Also, since $a(\tau)^2$ is a modular form of weight $2$ on
$\Gamma_0(3)$, $a(\tau)^2\big|\sigma_2=a(\tau)^2$. It follows that
$F(\tau_1,\tau_2)^2$ is a bimodular form of weight $2$ on
$(\Gamma_0(9),N(\Gamma_0(9)))$ with trivial characters. We summarize
our findings in the following proposition.

\begin{Proposition} Let $x(\tau_1,\tau_2)$ and $y(\tau_1,\tau_2)$ be
  defined by \eqref{equation: xy T3n} and $F(\tau_1,\tau_2)$ be the
  function in the last expression of \eqref{equation: cubic}. Then $x$
  and $y$ are bimodular functions and $F^2$ is a bimodular form of
  weight $2$ with trivial characters on
  $(\Gamma_0(9),N(\Gamma_0(9)))$. As a function of $x$ and $y$, $F$
  satisfies
  \begin{equation*}
    \begin{split}
      &3\theta_x\left(3\theta_x-2\theta_y\right)F
      -x\left(3\theta_x+2y+1\right)
      \left(3\theta_x+2\right)F, \\
      &\qquad
      -6x(2\theta_x+1)((1-4y)\theta_y+9y\theta_x)F=0, \\
      &\theta_y^2F-y(2\theta_y-3\theta_x)(2\theta_y-3\theta_x+1)F=0.
    \end{split}
  \end{equation*}
\end{Proposition}

\section{Ramanujan-type $1/\pi$-series from bimodular
  forms} \label{section: pi series}

In this section we shall describe procedures to obtain Ramanujan-type
$1/\pi$-series from bimodular forms.

\subsection{General form of $1/\pi$-series}
We first review the notion of CM-points on modular curves. Let $\O$ be
an order in $M(2,\Q)$, that is, a $\Z$-module of rank $4$ that is also
a subring with unity of $M(2,\Q)$. Let $\O^1$ be the group of elements
of determinant $1$ in $\O$. We call the quotient space
$\O^1\backslash\H$, after compactification by adding cusps, the
modular curve associated to $\O$ and denote it by $X(\O)$. Here we
will only consider two types of orders
$$
\left\{\M abcd\in M(2,\Z):N|c\right\}, \qquad
\left\{\M abcd\in M(2,\Z):N|c,a\equiv d\mod N\right\},
$$
where $N$ is a positive integer. The groups of
elements of determinant $1$ in the two cases are $\Gamma_0(N)$ and
$\Gamma_1(N)$ and the modular curves are $X_0(N)$ and $X_1(N)$,
respectively.

Let $\O$ be an order in $M(2,\Q)$. For an embedding
$\phi:\Q(\sqrt{d_0})\hookrightarrow M(2,\Q)$ from a
quadratic number field $\Q(\sqrt{d_0})$ to $M(2,\Q)$, it can be shown
that the intersection of $\phi(\Q(\sqrt{d_0}))$ and $\O$ is $\phi(R)$
for some quadratic order $R$ in $\Q(\sqrt{d_0})$. Let $d$ be the
discriminant of this quadratic order $R$. We say $\phi$ is an
\emph{optimal embedding} of discriminant $d$ with respect to $\O$. Now if $d_0<0$, then elements in $\phi(R)$ share a (unique)
common fixed point $\tau_\phi$ in $\H$. We call this point a
\emph{CM-point} of discriminant $d$ on the modular curve $X(\O)$.
Notice that if $\phi:\Q(\sqrt{d})\hookrightarrow M(2,\Q)$ is an
optimal embedding of discriminant $d$ with respect to $\O$, then so
is $-\phi$ and they share the same fixed point $\tau_\phi$. We say
$\phi$ is \emph{normalized} if
$$
\phi(\alpha)\begin{pmatrix}\tau_\phi\\1\end{pmatrix}
=\alpha\begin{pmatrix}\tau_\phi\\1\end{pmatrix}
$$
for all $\alpha\in\Q(\sqrt d)$. It is clear that this condition
amounts to the assumption that the $(2,1)$-entry of $\phi(\sqrt d)$
is positive.

\begin{Lemma} \label{lemma: one variable pi}
  Let $t$ be a nonconstant modular function on a modular
  curve $X$ assoicated to some order $\O$ in $M(2,\Q)$.
  Let $g=t'/2\pi i=\theta_qt$. Let $\tau_d$ be a CM-point of
  discriminant $d$ on $X$ with the corresponding normalized optimal
  embedding $\phi$. Let $\alpha=\phi(\sqrt d)$ and
  $h=(g|_2\alpha)/g$. We have
  \begin{equation} \label{equation: one variable pi}
  \frac{\theta_qg}{\theta_qt}(\tau_d)-\frac12(\theta_qh)(\tau_d)
  =\frac1{2\pi\Im\tau_d}
  \end{equation}
\end{Lemma}

\begin{proof} The proof is essentially the same as that of Equation
  (4) in \cite{Yang-Ramanujan}. For convenience of the reader, we
  reproduce the proof here.

  Write
  $$
  \alpha=\M{b_1}{b_2}{b_3}{b_4}.
  $$
  Since $\alpha$ is the image of $\sqrt d$ under the optimal embedding
  $\phi$, we have $\tr\alpha=0$ and $\det\alpha=|d|$. As $\tau_d$ is a
  fixed point of $\alpha$, it follows that
  $$
  \tau_d=\frac{(b_1-b_4)+\sqrt{(b_1-b_4)^2+4b_2b_3}}{2b_3}
  =\frac{b_1+\sqrt d}{b_3}
  $$
  (note that $\phi$ is normalized, so $b_3>0$)
  and
  \begin{equation} \label{equation: pi 1}
  \frac{\det\alpha}{(b_3\tau_d+b_4)^2}=\frac{|d|}{d}=-1.
  \end{equation}
  Consequently, by the definition of the slash operator,
  \begin{equation} \label{equation: pi 2}
  h(\tau_d)=\frac{\det\alpha}{(b_3\tau_d+b_4)^2}
  \frac{g(\alpha\tau_d)}{g(\tau_d)}=-1.
  \end{equation}

  Now recall that the Shimura-Maass operator $\partial_k$ of weight
  $k$ for a smooth function $f:\H\to\C$ and an integer $k$ is defined
  to be
  $$
  (\partial_kf)(\tau):=\frac1{2\pi i}\left(f'(\tau)+
    \frac{kf(\tau)}{\tau-\overline\tau}\right).
  $$
  The Shimura-Maass operators satisfy
  \begin{equation} \label{equation: Shimura 1}
  \partial_{k_1+k_2}(f_1f_2)=(\partial_{k_1}f_1)f_2+f_1
  (\partial_{k_2}f_2)
  \end{equation}
  and
  \begin{equation} \label{equation: Shimura 2}
  \partial_k(f|_k\gamma)=(\partial_kf)|_{k+2}\gamma
  \end{equation}
  for all smooth functions $f_1,f_2,f:\H\to\C$, any integers
  $k_1,k_2,k$, and any $\gamma\in\GL^+(2,\R)$ (see Equations (1.5) and
  (1.8) of \cite{Shimura-Maass}). The second property in particular
  implies that if $f$ is modular of weight $k$ on a subgroup $\Gamma$
  of $\GL^+(2,\R)$, then $\partial_kf$ is modular of weight $k+2$ on
  $\Gamma$.

  Now consider $\partial_2(g|_2\alpha)/g$. By \eqref{equation:
    Shimura 2}, it is equal to $(\partial_2g)|_4\alpha$. On the other
  hand, by \eqref{equation: Shimura 1},
  $$
  \frac{(\partial_2g)|_4\alpha}g=\frac{\partial_2(g|_2\alpha)}g
  =\frac{\partial_2(gh)}g
  =h\frac{\partial_2g}g+\partial_0h
  =h\frac{\partial_2g}g+\theta_qh.
  $$
  Evaluating the two sides at $\tau_d$ and using \eqref{equation: pi
    1} and \eqref{equation: pi 2}, we find that
  \begin{equation} \label{equation: pi 3}
  \frac{\partial_2g}g(\tau_d)=-\frac{\partial_2g}g(\tau_d)
  +(\theta_qh)(\tau_d).
  \end{equation}
  Since
  $$
  (\partial_2g)(\tau)=(\theta_qg)(\tau)+\frac{2g(\tau)}
  {2\pi i(\tau-\overline\tau)}
  =(\theta_qg)(\tau)-\frac{g(\tau)}{2\pi\Im\tau},
  $$
  \eqref{equation: pi 3} can be written as
  $$
  \frac{\theta_qg}g(\tau_d)-\frac12(\theta_qh)(\tau_d)
  =\frac1{2\pi\Im\tau_d}.
  $$
  This completes the proof of the lemma.
\end{proof}

Then the general form of $1/\pi$-series can be stated as follows.

\begin{Theorem} \label{theorem: general}
  Let $\{u_n\}$ be one of the six sporadic Ap\'ery-like
  sequences with parameters $(a,b,c)$, congruence subgroup $\Gamma$,
  modular function $t$, and modular form $f$ given in Table
  \ref{table: sporadic}. Let $x(X,Y)$ and $y(X,Y)$ be defined by
  \eqref{equation: Apery xy}. For two CM-points $\tau_1$ and $\tau_2$
  in $\H$ of discriminiants $d_1$ and $d_2$, respectively, such that
  $\Q(\sqrt{d_1})=\Q(\sqrt{d_2})$, set
  $$
  t_1=t(\tau_1), \quad t_2=t(\tau_2), \quad x_0=x(t_1,t_2), \quad
  y_0=y(t_1,t_2).
  $$
  Assume that the series $\sum_n\sum_{m\le\lfloor n/2\rfloor}u_n
  \binom{2m}m\binom n{2m}x_0^ny_0^m$ converges absolutely.
  Then there exist algebraic numbers $B_j$ and $C_j$, $j=1,2$, such
  that
  \begin{equation} \label{equation: general 1}
  \sum_{n=0}^\infty\sum_{m=0}^{\lfloor n/2\rfloor}u_n
  \binom{2m}m\binom n{2m}(n+B_1)x_0^ny_0^m
  =\frac{C_1}\pi
  \end{equation}
  and
  \begin{equation} \label{equation: general 2}
  \sum_{n=0}^\infty\sum_{m=0}^{\lfloor n/2\rfloor}u_n
  \binom{2m}m\binom n{2m}(m+B_2)x_0^ny_0^m
  =\frac{C_2}\pi.
  \end{equation}
  To be more concrete, let $\phi_j$ be the optimal embeddings defining
  the CM-points $\tau_j$ and $\alpha_j=\phi_j(\sqrt{d_j})$, $j=1,2$.
  Set
  $$
  g=\theta_qt, \qquad h_j=(g|_2\alpha_j)/g,
  $$
  and
  $$
  \delta_j=\frac{\theta_qh_j}{f^2}(\tau_j).
  $$
  Set also $\epsilon=f(\tau_2)/f(\tau_1)$. Then
  \begin{equation*}
    \begin{split}
      B_1&=\frac{c(t_1(\theta_xY)|_{(t_1,t_2)}
        +t_2(\theta_xX)|_{(t_1,t_2)})}
      {1-ct_1t_2}
      +(\theta_xX)|_{(t_1,t_2)}\frac{-\delta_1+2-4at_1+6ct_1^2}
      {4t_1(1-at_1+ct_1^2)} \\
      &\qquad\qquad+(\theta_xY)|_{(t_1,t_2)}
      \frac{-\delta_2+2-4at_2+6ct_2^2}{4t_2(1-at_2+ct_2^2)}.
    \end{split}
  \end{equation*}
  and
  $$
  C_1=\frac{\epsilon(1-ct_1t_2)(\theta_xX)|_{(t_1,t_2)}}
  {4t_1(1-at_1+ct_1^2)\Im\tau_1}
  +\frac{\epsilon^{-1}(1-ct_1t_2)(\theta_xY)|_{(t_1,t_2)}}
  {4t_2(1-at_2+ct_2^2)\Im\tau_2},
  $$
  where $\theta_x=x\partial/\partial x$, and similar expressions with
  $\theta_x$ replaced by $\theta_y:=y\partial/\partial y$ hold for
  $B_2$ and $C_2$.
\end{Theorem}

We remark that
$$
(\theta_xX)(X,Y)=\frac{X(1-cY^2)(1-aX+cX^2)}
{1-a(X+Y)+c(X^2+4XY+Y^2)-acXY(X+Y)+c^2X^2Y^2}
$$
and a similar expression with the roles of $X$ and $Y$ switched holds
for $(\theta_xY)(X,Y)$. Note that if we let $\alpha$ and $\beta$ be
the two numbers such that $1-aX+cX^2=(1-\alpha X)(1-\beta X)$, then
the denominator of the expression above can be factorized as
$$
(1-\alpha(X+Y)+cXY)(1-\beta(X+Y)+cXY).
$$
The expressions for $\theta_yX$ and $\theta_yY$ are more complicated.
We have
\begin{equation*}
  \begin{split}
  (\theta_yX)(X,Y)&=\frac{XY(1-aX+cX^2)(1-aY+cY^2)}{(Y-X)(1-cXY)^2}\\
  &\times\frac{1-2aX+3cX^2+cXY(3-2aX+cX^2)}
  {1-a(X+Y)+c(X^2+4XY+Y^2)-acXY(X+Y)+c^2X^2Y^2}.
  \end{split}
\end{equation*}

\begin{proof}[Proof of Theorem \ref{theorem: general}]
  The proof consists mostly of straightforward calculus. Abusing
  notations, for variables $\tau_j\in\H$, $j=1,2$, we let
  $t_j=t(\tau_j)$, $f_j=\sum_{n=0}^\infty u_nt_j^n$, $g_j=g(\tau_j)$,
  $x=x(t_1,t_2)$, and $y=y(t_1,t_2)$. By \eqref{equation: WZ
    identity},
  $$
  \sum_{n=0}^\infty\sum_{m=0}^{\lfloor n/2\rfloor}u_n\binom{2m}m
  \binom n{2m}x^ny^m=(1-ct_1t_2)f_1f_2.
  $$
  Applying the differential operator $\theta_x:=x\partial/\partial x$
  on the two sides and using the chain rule, we obtain
  \begin{equation} \label{equation: main theorem tmp}
    \begin{split}
      \sum_{n=0}^\infty\sum_{m=0}^{\lfloor n/2\rfloor}
      u_n\binom{2m}m\binom n{2m}nx^ny^m
      &=-c(t_1\theta_xt_2+t_2\theta_xt_1)f_1f_2 \\
      &\qquad+(1-ct_1t_2)\left(f_2\frac{df_1}{dt_1}\theta_xt_1
      +f_1\frac{df_2}{dt_2}\theta_xt_2\right).
    \end{split}
  \end{equation}
  We check case by case that the modular forms $f$ and $g$ are
  related by
  $$
  g=f^2t(1-at+ct^2).
  $$
  Hence,
  $$
  f\frac{df}{dt}=\frac1{2t(1-at+ct^2)}\frac{dg}{dt}
  -f^2\frac{1-2at+3t^2}{2t(1-at+ct^2)}.
  $$
  Plugging this into \eqref{equation: main theorem tmp}, evaluating at
  two CM-points $\tau_1$ and $\tau_2$ of discriminants $d_1$ and $d_2$
  with $\Q(\sqrt{d_1})=\Q(\sqrt{d_2})$, and setting
  $\epsilon=f(\tau_2)/f(\tau_1)$, we obtain
  \begin{equation*}
    \begin{split}
      &\sum_{n=0}^\infty\sum_{m=0}^{\lfloor n/2\rfloor}
      u_n\binom{2m}m\binom n{2m}nx^ny^m
      =-c(t_1\theta_xt_2+t_2\theta_xt_1)f_1f_2 \\
      &\qquad+\epsilon(1-ct_1t_2)\theta_xt_1
      \left(\frac1{2t_1(1-at_1+ct_1^2)}\frac{dg_1}{dt_1}
        -f_1^2\frac{1-2at_1+3ct_1^2}{2t_1(1-at_1+ct_1^2)}\right) \\
      &\qquad+\frac1\epsilon(1-ct_1t_2)\theta_xt_2
      \left(\frac1{2t_2(1-at_2+ct_2^2)}\frac{dg_2}{dt_2}
        -f_2^2\frac{1-2at_2+3ct_2^2}{2t_2(1-at_2+ct_2^2)}\right).
    \end{split}
  \end{equation*}
  Finally, substituting \eqref{equation: one variable pi} into the
  last expression and simplifying, we obtain the claimed formula.
\end{proof}

\begin{Remark}\label{rem-thm-Sun} Following the same computation, we can also obtain the
  general form of $1/\pi$-series for $u_n=(a)_n(1-a)_n/(n!)^2$ with
  $a\in \{1/2,1/3,1/4\}$.

  Let $t(\tau), x(\tau_1,\tau_2)$ and
  $y(\tau_1,\tau_2)$ be given in Proposition \ref{prop-hyper},
  and assume other notations and conditions in Theorem \ref{theorem:
    general}.
  Then \eqref{equation: general 1} and \eqref{equation: general 2}
  hold with
  \begin{align*}
      B_1&=\frac{(1-t_1)(1-t_2)}{4(1-t_1t_2)}(4-\delta_1-\delta_2), \\
      C_1&=\frac{(1-t_1)(1-t_2)}{4(1-t_1t_2)}\left(\frac{\epsilon}{\Im \tau_1}+\frac{\epsilon^{-1}}{\Im \tau_2} \right).
  \end{align*}
  and
\begin{align}
 B_2=\frac{1}{4(1-t_1t_2)(t_2-t_1)}\left((2-\delta_1)(1-t_1^2)t_2-(2-\delta_2)
 (1-t_2^2)t_1 \right), \\
 C_2=\frac{1}{4(1-t_1t_2)(t_2-t_1)}\left(\frac{\epsilon}{\Im \tau_1}(1-t_1^2)t_2-\frac{\epsilon^{-1}}{\Im \tau_2}
 (1-t_2^2)t_1 \right).
 \end{align}
\end{Remark}

\subsection{$1/\pi$-series with rational $x$ and $y$}

In this section, we shall describe situations where the values of the
bimodular functions $x$ and $y$ are rational numbers.

Let $\{u_n\}$ be one of the sporadic Ap\'ery-like sequences and let
the notations $(a,b,c)$, $\alpha$, $\beta$, $\Gamma$, $G$, $t$, $x$,
and $y$ be given as in Proposition \ref{proposition: sporadic
  bimodular}.
\medskip

\paragraph{\bf Case $\Gamma=\Gamma_0(6)$}
Consider the first three cases where
$\Gamma=\Gamma_0(6)$ and $G=N(\Gamma_0(6))$. For a negative
discriminant $d$, we let $\CM(d)$ denote the set of CM-points of
discriminant $d$ on $X_0(6)$. It is clear that the group of
Atkin-Lehner involutions acts on $\CM(d)$.

\begin{Lemma}[{\cite[Theorem 2]{Ogg}}] Let $d$ be a negative
  discriminant. Write it as $d=r^2d_0$, where $d_0$ is a fundamental
  discriminant. Then the number $|\CM(d)|$ of CM-points of
  discriminant $d$ on $X_0(6)$ is
  $$
  |\CM(d)|=h(d)\times\prod_{p|6}\nu_p(d),
  $$
  where $h(d)$ is the class number of the quadratic order of
  discriminant $d$ and $\nu_p(d)$ is defined by
  $$
  \nu_p(d)=\begin{cases}
    \displaystyle 1+\JS{d}p, &\text{if }p\nmid r, \\
    2, &\text{if }p|r. \end{cases}
  $$
\end{Lemma}

Note that the function $\nu_p(d)-1$ is called the Eichler symbol (in
the case $M(2,\Q)$) in literature.

\begin{Lemma} Let $d$ be a negative discriminant such that the set
  $\CM(d)$ of CM-points of discriminant $d$ on $X_0(6)$ is nonempty.
  Let $K=\Q\sqrt d)$, and $R$, $h(d)$, and $H$ be the quadratic order
  of discriminant $d$, its class number, and its ring class field,
  respectively.
  \begin{enumerate}
    \item[(i)] Assume that $6|d$. Then
      $\Q(t(\tau))=H\cap\R$ for any $\tau\in\CM(d)$.
    \item[(ii)] Assume that $6\nmid d$. Then $\Q(t(\tau))=H$ for any
      $\tau\in\CM(d)$.
  \end{enumerate}

  Moreover, assume that the class group of $R$ is an elementary
  $2$-group, i.e., that each genus of the class group consists of one
  single class. Let $W$ be the full Atkin-Lehner group on $X_0(6)$.
  If $|\CM(d)|$ is $2$ or $4$, then the values of
  $x(\tau_1,\tau_2)$ and $y(\tau_1,\tau_2)$ are both rational
  numbers for any $\tau_1,\tau_2\in\CM(d)$; if $|\CM(d)|=8$, then the
  same conclusion holds for any $\tau_1,\tau_2\in\CM(d)$ such that
  $\tau_2$ is not in the $W$-orbit of $\tau_1$.
\end{Lemma}


\begin{proof} The assertion about $|\CM(d)|$ follows from the previous
  lemma, while that about $\Q(t(\tau))$ follows from Theorem 5.2 of
  \cite{Gonzalez-Rotger}. For the assertion about
  rationality of $x$ and $y$, here we will only prove the case
  $|\CM(d)|=8$ since the proof of the case $|\CM(d)|=2$ or $|\CM(d)|=4$
  is similar and simpler.



  Assume that $|\CM(d)|=8$ and the class group of $R$ is an elementary
  $2$-group. When $6|d$, by Lemma 5.9 of
  \cite{Gonzalez-Rotger}, the action of each element of $W$ coincides
  with that of some element of $\Gal(H/K)$. Say, $G$ is the subgroup
  of $\Gal(H/K)$ corresponding to $W$. Let $\tau_1$ and $\tau_2$
  be elements in $\CM(d)$ such that $\tau_2$ is not in the $W$-orbit
  of $\tau_1$. Consider $x(\tau_1,\tau_2)$. Since
  $x(w\tau_1,w\tau_2)=x(\tau_1,\tau_2)$ for all $w\in W$, we have
  $x(\tau_1,\tau_2)^\sigma=x(\tau_1,\tau_2)$ for all $\sigma\in W$.
  Let $\sigma'$ be the element in $\Gal(H/K)$ such that
  $t(\tau_1)^{\sigma'}=t(\tau_2)$. Since the class group is assumed to
  be an elementary $2$-group, we also have
  $t(\tau_2)^{\sigma'}=t(\tau_1)$. It follows that
  $x(\tau_1,\tau_2)^{\sigma'}=x(\tau_2,\tau_1)=x(\tau_1,\tau_2)$ as
  well. Therefore, $x(\tau_1,\tau_2)$ is invariant under
  $\Gal(H/K)$. By the first half of the lemma, we have
  $x(\tau_1,\tau_2)\in K\cap\R=\Q$. The same conclusion holds for
  $y(\tau_1,\tau_2)$. When $(d,6)=3$, by Lemmas 5.9 and 5.10 of
  \cite{Gonzalez-Rotger}, the action of $w_3$ coincides with some
  element of $\Gal(H/K)$ and that of $w_2$ coincides with complex
  conjugation. By a similar reasoning as in the case $6|d$, we see
  that $x(\tau_1,\tau_2)$ and $y(\tau_1,\tau_2)$
  are rational numbers for all $\tau_1,\tau_2\in\CM(d)$ such that
  $\tau_2$ is not in the $W$-orbit of $\tau_1$.
  The proof of the remaining cases is omitted.
\end{proof}

\begin{Example} \label{example: 6}
\begin{enumerate}
  \item[(i)] The class group of the quadratic order of
    discriminant $-420$ is an elementary $2$-group of order $8$.
    We represent the CM-points $\tau$ by their corresponding quadratic
    forms $[A,B,C]:=Ax^2+Bxy+Cy^2$ (i.e.,
    $\tau=(-B+\sqrt{B^2-4AC})/2A$). The $8$ CM-points of
    discriminant $-420$ are represented by
    \begin{equation*}
      \begin{split}
        Q_1=[6,6,19],~Q_2=[30,30,11],&
        ~Q_3=[42,42,13],~Q_4=[66,30,5],\\
        Q_5=[78,42,7],~Q_6=[114,6,1],&
        ~Q_7=[210,210,53],~Q_8=[318,210,35]
      \end{split}
    \end{equation*}
    The actions of $w_2$, $w_3$, and $w_6$ map $Q_1$ to $Q_8$, $Q_7$,
    and $Q_6$, respectively. Thus, we choose
    \begin{equation} \label{equation: tau 6}
    \tau_1=\frac{-6+\sqrt{-420}}{12}, \qquad
    \tau_2=\frac{-30+\sqrt{-420}}{60},
    \end{equation}
    corresponding to $Q_1$ and $Q_2$, respectively. Since
    $-420=-3\times140=-7\times 60=-20\times 21$, we expect that the
    values of $t$ at CM-points of discriminant $-420$ belong to
    $\Q(\sqrt3,\sqrt5,\sqrt7)$. Indeed, for the case
    $(a,b,c)=(-17,-6,72)$ and
    $t(\tau)=\eta(2\tau)\eta(6\tau)^5/\eta(\tau)^5\eta(3\tau)$, we
    find
    \begin{equation*}
      \begin{split}
    t(\tau_1)&=\frac{3\sqrt3-5}{12}\cdot\frac{5-\sqrt{21}}2
    \cdot\frac{\sqrt7-\sqrt3}2\cdot
    (\sqrt{21}-2\sqrt5)(4-\sqrt{15}), \\
    t(\tau_2)&=\frac{3\sqrt3+5}{12}\cdot\frac{5-\sqrt{21}}2
    \cdot\frac{\sqrt7-\sqrt3}2\cdot
    (\sqrt{21}-2\sqrt5)(4+\sqrt{15}),
      \end{split}
    \end{equation*}
    and
    $$
      x(\tau_1,\tau_2)=-\frac{71}{1008}, \qquad
    y(\tau_1,\tau_2)=\frac1{142^2}.
    $$
    They are indeed rational numbers.
  \item[(ii)] Let $d=-20$. The $4$ CM-points $\tau_j$, $j=1,\ldots,4$
    of discriminant $-20$ are represented by
    $$
    Q_1=[6,2,1],~Q_2=[6,-2,1],~Q_3=[18,14,3],~Q_4=[18,-14,3].
    $$
    The actions of $w_2$, $w_3$, and $w_6$ map $Q_1$ to $Q_3$, $Q_2$,
    and $Q_4$, respectively. In the case $(a,b,c)=(-17,-6,72)$, we
    have
    $$
    t(\tau_1)=\frac{-2+\sqrt5-i(2+\sqrt5)}{36}, \qquad
    t(\tau_3)=\frac{-2-\sqrt5+i(2-\sqrt5)}{36},
    $$
    $t(\tau_2)=\overline{t(\tau_1)}$, and
    $t(\tau_4)=\overline{t(\tau_3)}$. We find
    $$
    x(\tau_1,\tau_2)=\infty, \qquad
    x(\tau_1,\tau_3)=-\frac9{80}, \qquad
    x(\tau_1,\tau_4)=-\frac18,
    $$
    and
    $$
    y(\tau_1,\tau_2)=\frac14, \qquad
    y(\tau_1,\tau_3)=y(\tau_1,\tau_4)=\frac1{324}.
    $$
  \item[(iii)] Occasionally, there are $\tau_1\in\CM(d_1)$ and
    $\tau_2\in\CM(d_2)$ with $d_1\neq d_2$ and
    $\Q(\sqrt{d_1})=\Q(\sqrt{d_2})$ such that $x(\tau_1,\tau_2)$
    and $y(\tau_1,\tau_2)$ are rational numbers. For example, let
    $d_1=-12$ and $d_2=-48$. By the lemma above, $t(\tau)\in\Q$ for
    $\tau\in\CM(-12)$. Since $h(-48)=2$ and $|\CM(-48)|=4$, the
    Atkin-Lehner group $W$ acts on $\CM(-48)$ transitively. It follows
    that $x(\tau_1,\tau_2)$ and $y(\tau_1,\tau_2)$ are rational for
    $\tau_1\in\CM(-12)$ and $\tau_2\in\CM(-48)$. For example, let
    $$
    \tau_1=\frac{-3+\sqrt{-3}}6, \qquad
    \tau_2=\frac{\sqrt{-3}}6,
    $$
    with corresponding quadratic forms $[6,6,2]$ and $[12,0,1]$.
    In the case $(a,b,c)=(-17,-6,72)$, we have
    $$
    t(\tau_1)=-\frac1{12}, \qquad
    t(\tau_2)=\frac{1+\sqrt3}6,
    $$
    and
    $$
    x(\tau_1,\tau_2)=-\frac5{36}, \qquad
    y(\tau_1,\tau_2)=-\frac1{50}.
    $$
  \end{enumerate}
\end{Example}

\paragraph{\bf Case $\Gamma=\Gamma_0(8)$} We next consider the case
$(a,b,c)=(12,4,32)$ with $\Gamma=\Gamma_0(8)$ and $G$ generated by
$\Gamma$, $w_8$, and $\SM4184$. As before, we let $\CM(d)$ denote the
set of CM-points of discriminant $d$ on $X_0(8)$. For our purpose, we
need to know how $w_8$, $\SM4184$, and the complex conjugation act on
$\CM(d)$. In the subsequent discussion, we let $d$ be a negative
discriminant, $K=\Q(\sqrt d)$, and $R$, $h(d)$, and $H$ be the
quadratic order of discriminant $d$, its class number, and its ring
class field, respectively. Write $d=r^2d_0$, where $d_0$ is a
fundamental discriminant.

Recall that the Galois group $\Gal(H/K)$ acts on the set $\CM(d)$
through the Shimura reciprocity and each $\Gal(H/K)$-orbit contains
$h(d)$ points. Moreover, two CM-points lie in the same
$\Gal(H/K)$-orbit if and only if their normalized optimal embeddings
are locally equivalent at every finite place. Hence
$$
  |\CM(d)|=h(d)\prod_{p\text{ primes}}\nu_p(d)
$$
where $\nu_p(d)$ is the number of optimal embeddings of discriminant
$d$ from $K$ into $\O_p:=\SM\Z\Z{8\Z}\Z\otimes_\Z\Z_p$ modulo
conjugation by $\O^\times_p$ (see \cite[Section 5, Chapter
III]{Vigneras}). Since $\O_p=\SM{\Z_p}{\Z_p}{\Z_p}{\Z_p}$ for odd
primes $p$, we have $\nu_p(d)=1$ for odd primes $p$ and
$|\CM(d)|=\nu_2(d)h(d)$. We now determine $\nu_2(d)$ and describe the
$\Gal(H/K)$-orbits of $\CM(d)$ and how $w_8$, $\SM4184$, and the
complex conjugation act on them. (Note that one can obtain a formula
for $\nu_2(d)$ from Theorem 2 of \cite{Ogg}.)

\begin{Lemma} \label{lemma: actions 8}
  In the following, we represent a CM-point by the
  corresponding quadratic form $[A,B,C]:=Ax^2+Bxy+Cy^2$ with
  $B^2-4AC=d$, $8|A$, and $(A/8,B,C)=1$ (i.e., $\tau=(-B+\sqrt
  d)/2A$).
  \begin{enumerate}
    \item[(i)] If $d_0\equiv 5\mod 8$ and $4\nmid r$ or if $4|d_0$ and
      $2\nmid r$, then $|\CM(d)|=0$.
    \item[(ii)] Assume that $d\equiv 1\mod 8$. Let $b$ be an integer
      such that $b^2\equiv d\mod 32$. Then $\nu_2(d)=2$ and the two
      $\Gal(H/K)$-orbits of $\CM(d)$ are
      $$
      \{[A,B,C]:B\equiv b\mod 16\}, \qquad
      \{[A,B,C]:B\equiv -b\mod 16\}.
      $$
      The actions of $w_8$ and the complex conjugation swap the two
      orbits. The action of $\SM4184$ maps the two orbits to the first
      two orbits of $\CM(4d)$ in (iii).
    \item[(iii)] Assume that $d_0\equiv 1\mod 8$ and $2\|r$. Let
      $b_1$ and $b_2$ be integers such that
      $b_1^2-d_0\equiv 8\mod 16$ and $b_2^2-d_0\equiv 0\mod 16$. Then
      $\nu_2(d)=6$ and the six $\Gal(H/K)$-orbits are
      \begin{equation*}
        \begin{split}
          C_1:&\{[A,B,C]:B\equiv 2b_1\mod 16\}, \\
          C_2:&\{[A,B,C]:B\equiv -2b_1\mod 16\}, \\
          C_3:&\{[A,B,C]:B\equiv 2b_2\mod 16,~C\text{ even}\}, \\
          C_4:&\{[A,B,C]:B\equiv 2b_2\mod 16,~C\text{ odd}\}, \\
          C_5:&\{[A,B,C]:B\equiv -2b_2\mod 16,~C\text{ even}\}, \\
          C_6:&\{[A,B,C]:B\equiv -2b_2\mod 16,~C\text{ odd}\}.
        \end{split}
      \end{equation*}
      The action of $w_8$ interchanges $C_1$ with $C_2$, $C_3$ with
      $C_6$, and $C_4$ with $C_5$. The action of the complex
      conjugation interchanges $C_1$ with $C_2$, $C_3$ with $C_5$, and
      $C_4$ with $C_6$. The action of $\SM4184$ interchanges $C_3$
      with $C_5$, $C_4$ with $C_6$, and maps $C_1$ and $C_2$ to the
      two orbits of $\CM(d/4)$ in (ii).
    \item[(iv)] If $d_0$ is odd and $4\|r$, then $\nu_2(d)=4$ and
      the four $\Gal(H/K)$-orbits are
      \begin{equation*}
        \begin{split}
          C_1:&\{[A,B,C]:B\equiv 4\mod 16,~C\text{ even}\}, \\
          C_2:&\{[A,B,C]:B\equiv 4\mod 16,~C\text{ odd}\}, \\
          C_3:&\{[A,B,C]:B\equiv -4\mod 16,~C\text{ even}\}, \\
          C_4:&\{[A,B,C]:B\equiv -4\mod 16,~C\text{ odd}\}.
        \end{split}
      \end{equation*}
      The action of $w_8$ interchanges $C_1$ with $C_4$ and $C_2$ with
      $C_3$. The action of the complex conjugation interchanges $C_1$
      with $C_3$ and $C_2$ with $C_4$. The action of $\SM4184$
      interchanges $C_1$ with $C_2$ and $C_3$ with $C_4$.
    \item[(v)] If $4\|d_0$ and $2\|r$, then $\nu_2(d)=2$ and
      the two $\Gal(H/K)$-orbits are
      $$
      \{[A,B,C]:B\equiv 4\mod 16\}, \qquad
      \{[A,B,C]:B\equiv-4\mod 16\}.
      $$
      All $w_8$, $\SM4184$, and the complex conjugation switch the two
      orbits.
    \item[(vi)] If $8|d_0$ and $2\|r$, then $\nu_2(d)=2$ and the two
      $\Gal(H/K)$-orbits are
      $$
      \{[A,B,C]:B\equiv 0\mod 16\}, \qquad
      \{[A,B,C]:B\equiv 8\mod 16\}.
      $$
      All $w_8$, $\SM4184$, and the complex conjugation fix every
      orbit.
    \item[(vii)] If $64|d$, then $\nu_2(d)=4$ and the four
      $\Gal(H/K)$-orbits are
      \begin{equation*}
        \begin{split}
          C_1:&\{[A,B,C]:B\equiv 0\mod 16,~C\text{ even}\}, \\
          C_2:&\{[A,B,C]:B\equiv 0\mod 16,~C\text{ odd}\}, \\
          C_3:&\{[A,B,C]:B\equiv 8\mod 16,~C\text{ even}\}, \\
          C_4:&\{[A,B,C]:B\equiv 8\mod 16,~C\text{ odd}\}.
        \end{split}
      \end{equation*}
      The action of $w_8$ interchanges of $C_1$ with $C_2$ and $C_3$
      with $C_4$. The action of $\SM4184$ interchanges $C_1$ with
      $C_4$ and $C_2$ with $C_3$. The complex conjugation fixes every
      orbit.
  \end{enumerate}
  In particular,
  \begin{enumerate}
    \item[(i)] if $32\nmid d$, then $\Q(t(\tau))=H$ for all
      $\tau\in\CM(d)$, and
    \item[(ii)] if $32|d$ and the class group of $R$ is an elementary
      $2$-group, then $\Q(t(\tau))=H\cap\R$ for all
      $\tau\in\CM(d)$.
  \end{enumerate}
\end{Lemma}

\begin{proof}
  For $a\in\Z_2$, let $v_2(a)$ denote the $2$-adic valuation of $a$.
  Let $\phi:R\to\O_2$, $\O_2=\SM{\Z_2}{\Z_2}{8\Z_2}{\Z_2}$,
  be an optimal embedding of
  discriminant $d$. We remark that $\phi$ is completely determined by
  $\phi(\sqrt d)$.

  Consider the case $d$ is odd first. We have
  $$
  \phi(\sqrt d)=\M ab{8c}{-a}
  $$
  for some $a,b,c\in\Z_2$ with $a^2+8bc=d$. It is clear that if $d$ is
  not congruent to $1$ modulo $8$, then no such $a$, $b$, $c$ exist
  and hence $\nu_2(d)=0$. Now assume that $d\equiv 1\mod 8$. Note that
  since $\phi$ is an optimal embedding of discriminant $d$, we must
  have $\phi((1+\sqrt d)/2)\in\O_2$. In other words, we have $2|b,c$
  and $a^2\equiv d\mod 32$. Now we check by direct computation the
  following properties.
  \begin{enumerate}
  \item[(i)] Among the residue classes modulo $16$, there are
    precisely two residue classes $a$ such that $a^2\equiv d\mod 32$.
  \item[(ii)] There exists $k\in\Z_2$ such that the $2$-adic valuation
    of the $(1,2)$-entry of
    $$
    \M1k01\M ab{8c}{-a}\M1{-k}01
    $$
    is $1$. In other words, $\phi$ is $\O_2^\times$-equivalent to the
    one given by $\sqrt d\mapsto\SM ab{8c}{-a}$ with $a^2\equiv d\mod
    32$, $2\|b$ and $2|c$, which we assume from now on.
  \item[(iii)] Let $\SM ab{8c}{-a}$ be the matrix from (ii). We have
    $$
    \M{b/2}001^{-1}\M ab{8c}{-a}\M{b/2}001=\M a2{8bc}{-a}.
    $$
    Hence $\phi$ is $\O_2^\times$-equivalent to $\sqrt d\mapsto\SM
    a2{(d-a^2)/2}{-a}$.
  \item[(iv)] Two optimal embeddings $\sqrt
    d\to\SM{a_1}2{(d-a_1^2)/2}{-a_1}$ and $\sqrt
    d\to\SM{a_2}2{(d-a_2^2)/2}{-a_2}$ are $\O_2^\times$-equivalent if
    and only if $a_1\equiv a_2\mod 16$.
  \end{enumerate}
  Based on these properties, we conclude that if $d\equiv 1\mod 8$,
  then there are two $\Gal(H/K)$-orbits of CM-points of discriminant
  $d$ and two points in $\CM(d)$ are in the same $\Gal(H/K)$-orbit if
  and only if their corresponding quadratic forms
  $A_1x^2+B_1xy+C_1y^2$ and $A_2x^2+B_2xy+C_2y^2$ satisfy $B_1\equiv
  B_2\mod 16$.

  We next consider the case $4|d$. We must have $\phi(\sqrt
  d/2)\in\O_2$, say
  $$
  \phi(\sqrt d/2)=\M ab{8c}{-a}
  $$
  for some $a,b,c\in\O_2$ satisfying $a^2+8bc=d/4$.
  If $d/4$ is odd, then $a$ must be odd. If $d/4\equiv 5\mod 8$, no
  such $a$, $b$, and $c$ can exist and hence $\nu_2(d)=0$. When
  $d/4\equiv 1\mod 8$, it is clear that for each odd element $a$ of
  $\O_2$, there are $b$ and $c$ in $\O_2$ such that $a^2+8bc=d/4$.
  We now check the following properties.
  \begin{enumerate}
    \item[(i)] Among the $4$ odd residue classes modulo $8$, there are
      exactly $2$ classes $a$ such that $8\|(d/4-a^2)$ and the other
      two classes satisfy $16|(d/4-a^2)$.
    \item[(ii)] In the case $8\|(d/4-a^2)$, $\phi$ is
      $\O_2^\times$-equivalent to $\sqrt d/2\mapsto\SM
      a1{d/4-a^2}{-a}$. In the case $16|(d/4-a^2)$, $\phi$ is
      $\O_2^\times$-equivalent to $\sqrt d/2\mapsto\SM
      a1{d/4-a^2}{-a}$ or $\sqrt d/2\mapsto\SM a{(d/4-a^2)/8}8{-a}$
      and the two are inequivalent.
    \item[(iii)] Two optimal embeddings $\sqrt
      d/2\mapsto\SM{a_1}1{(d/4-a_1^2)}{-a_1}$ and $\sqrt
      d/2\mapsto\SM{a_2}1{(d/4-a_2^2)}{-a_2}$ are
      $\O_2^\times$-equivalent if and only if $a_1\equiv a_2\mod 8$.
      The same property also holds for embeddings of the form $\sqrt
      d/2\mapsto\SM a{(d/4-a^2)/8}8{-a}$.
  \end{enumerate}
  It follows that $\nu_2(d)=6$. Also, two CM-points of discriminant
  $d$ are in the same $\Gal(H/K)$-orbit if and only if their
  corresponding quadratic forms $A_1x^2+B_1xy+C_1y^2$ and
  $A_2x^2+B_2xy+C_2y^2$ satisfy $B_1\equiv B_2\mod 8$ and $C_1\equiv
  C_2\mod 2$.

  The proof of the remaining cases is similar. We skip the details
  and just summarize our findings as follows.
  \begin{enumerate}
    \item[(i)] If $8\|d$, no $a,b,c\in\Z_2$ can satisfy
      $a^2+8bc=d/4$. Thus, $\nu_2(d)=0$.
    \item[(ii)] If $16\|d$ and $d/16\equiv 3\mod 4$ (i.e. $2\|r$),
      then any optimal embedding $\phi$ is $\O_2^\times$-equivalent
      to $\sqrt d/2\mapsto\M21{d/4-4}{-2}$ or $\sqrt
      d/2\mapsto\M{-2}1{d/4-4}{2}$ and the two are not
      $\O_2^\times$-equivalent. (Note that $8\|(d/4-4)$.) Hence
      $\nu_2(d)=2$.
    \item[(iii)] If $32\|d$, then $2\|r$ and any optimal embedding
      is $\O_2^\times$-equivalent to $\sqrt d/2\mapsto\SM01{d/4}0$
      or $\sqrt d/2\mapsto\SM41{d/4-16}{-4}$ and the two are not
      $\O_2^\times$-equivalent. Hence $\nu_2(d)=2$.
    \item[(iv)] If $64|d$, then $4|r$ and any optimal embedding is
      $\O_2^\times$-equivalent to one of $\sqrt
      d/2\mapsto\SM01{d/4}0$, $\sqrt d/2\mapsto\SM0{d/32}80$, $\sqrt
      d/2\mapsto\SM41{d/4-16}{-4}$, and $\sqrt
      d/2\mapsto\SM4{d/32-2}8{-4}$. Hence $\nu_2(d)=4$.
  \end{enumerate}

  Now the action $w_8$ maps the quadratic form $[A,B,C]$ to
  $[8C,-B,A/8]$ and the action of the complex conjugation maps
  $[A,B,C]$ to $[A,-B,C]$. Finally, we compute that
  $$
  \M4184\M ab{8c}{-a}\M4184^{-1}
  =\M{3a-4b+4c}{-a+2b-c}{8a-8b+16c}{-3a+4b-4c}.
  $$
  From these, we easily obtain our description of the actions in each
  case. Note that when $R$ is an elementary $2$-group, the ring class
  field $H$ is of the form $\Q(\sqrt d,\sqrt{a_1},\ldots,\sqrt{a_k})$
  for some positive rational numbers $a_1,\ldots,a_k$. In the case
  $32|d$, since the complex conjugation fixes every $\Gal(H/K)$-orbit
  and $\Gal(H/K)$ acts transitively on each orbit, we must have
  $t(\tau)\in\Q(\sqrt{a_1},\ldots,\sqrt{a_k})$ for all
  $\tau\in\CM(d)$.
This complete the proof of the lemma.
\end{proof}

\begin{Corollary} \label{corollary: rationality 8}
  Let $d$ be a negative discriminant such that the set
  $\CM(d)$ of CM-points of discriminant $d$ on $X_0(8)$ is
  nonempty and write $d=r^2d_0$, where $d_0$ is a fundamental
  discriminant. Let $R$, $h(d)$, and $H$ be the quadratic order of
  discriminant $d$, its class number, and its ring class field,
  respectively. Let $W$ be the subgroup of
  $N(\Gamma_0(8))/\Gamma_0(8)$ generated by the Atkin-Lehner
  involution $w_8$ and $\SM4184$.
  \begin{enumerate}
  \item[(i)] Assume that $d\equiv 1\mod 8$ and $h(d)=1$ (i.e.,
    $d=-7$). Let $\tau_1$ and $\tau_2$ be any two points in $\CM(d)$
    and the first two orbits of $\CM(4d)$ described in the previous
    lemma, or let $\tau_1$ and $\tau_2$ be any two points in the last
    four orbits of $\CM(4d)$, or let $\tau_1$ and $\tau_2$ be any two
    points in $\CM(16d)$ such that $\tau_2$ is not in the $W$-orbit of
    $\tau_1$.
  \item[(ii)] Assume that $d\equiv 1\mod 8$ and $h(d)=2$. Let $\tau_1$
    and $\tau_2$ be any two points in $\CM(d)$ and the first two
    orbits of $\CM(4d)$ described in the previous lemma such that
    $\tau_2$ is not in the $W$-orbit of $\tau_1$, or let $\tau_1$ and
    $\tau_2$ be any two points in the last four orbits of $\CM(4d)$
    such that $\tau_2$ is not in the $W$-orbit of $\tau_1$.
  \item[(iii)] Assume that $4\|d_0$ and $2\|r$. Assuming that $h(d)=1$
    or $h(d)=2$, let $\tau_1$ and $\tau_2$ be any two points in
    $\CM(d)$, or assuming that the class group of $R$ is the Klein
    $4$-group, let $\tau_1$ and $\tau_2$ be any two points in $\CM(d)$
    such that $\tau_2$ is not in the $W$-orbit of $\tau_1$.
  \item[(iv)] Assume that $8|d_0$ and $2\|r$. Assuming that the class
    group of $R$ is an elementary $2$-group of order $2$ or $4$,
    let $\tau_1$ and $\tau_2$ be any two points in the same
    $\Gal(H/K)$-orbit in $\CM(d)$, or
    assuming that the class group of
    $R$ is an elementary $2$-group of order $8$, let $\tau_1$ and
    $\tau_2$ be two points in the same $\Gal(H/K)$-orbit in $\CM(d)$
    such that $\tau_2$ is not in the $W$-orbit of $\tau_1$.
  \item[(v)] Assume that $d=-64$. Let $\tau_1$ and $\tau_2$ be any two
    points in $\CM(-64)$ such that $\tau_2$ is not in the $W$-orbit of
    $\tau_1$.
  \end{enumerate}
  Then the values of $x(\tau_1,\tau_2)$ and $y(\tau_1,\tau_2)$ are
  rational numbers.
\end{Corollary}

\begin{Example}
  \begin{enumerate}
  \item[(i)] Let $d=-112$. The four $\Gal(H/K)$-orbits of $\CM(-112)$
    are
    \begin{equation*}
      \begin{split}
      \{Q_1=[8,4,4],~Q_2=[56,-28,4]\}, \quad
      &\{Q_3=[32,4,1],~Q_4=[32,-28,7]\}, \\
      \{Q_5=[8,-4,4],~Q_6=[56,28,4]\}, \quad
      &\{Q_7=[32,-4,1],Q_8=[32,28,7]\}
      \end{split}
    \end{equation*}
    Let $\tau_j$, $j=1,\ldots,8$, be the corresponding CM-points.
    The actions of $w_8$, $\SM4184$, and the complex conjugation map
    $Q_1$ to $Q_7$, $Q_4$, and $Q_5$, respectively. The $W$-orbit of
    $Q_1$ is $\{Q_1,Q_4,Q_6,Q_7\}$. We expect that $x(\tau_1,\tau_j)$
    and $y(\tau_1,\tau_j)$ will be rational for $j=2,3,5,8$. Indeed,
    we find
    $$
    x(\tau_1,\tau_2)=\frac{16}{63}, \quad
    x(\tau_1,\tau_3)=\frac1{8}, \quad
    x(\tau_1,\tau_5)=-\frac1{252}, \quad
    x(\tau_1,\tau_8)=\frac1{8},
    $$
    and
    $$
    y(\tau_1,\tau_2)=y(\tau_1,\tau_3)=\frac1{256}, \quad
    y(\tau_1,\tau_5)=y(\tau_1,\tau_8)=16.
    $$
    The case $j=5$ will yield a $1/\pi$-series.
  \item[(ii)] Let $d=-15$. The two $\Gal(H/K)$-orbits of $\CM(-15)$
    are
    $$
    \{Q_1=[8,7,2],~Q_2=[16,7,1]\}, \quad
    \{Q_3=[8,-7,2],~Q_4=[16,-7,1]\}.
    $$
    The first two $\Gal(H/K)$-orbits of $\CM(-60)$ are
    $$
    \{Q_5=[8,6,3],~Q_6=[24,6,1]\}, \quad
    \{Q_7=[8,-6,3],~Q_8=[24,-6,1]\}.
    $$
    Let $\tau_j$ be the corresponding CM-points.
    The actions of $w_8$, $\SM4184$, and the complex conjugation map
    $Q_1$ to $Q_4$, $Q_7$, and $Q_3$, respectively, and the $W$-orbit
    of $Q_1$ is $\{Q_1,Q_4,Q_6,Q_7\}$. Hence, $x(\tau_1,\tau_j)$ and
    $y(\tau_1,\tau_j)$ should be rational numbers for $j=2,3,5,8$.
    Indeed, we have
    $$
    x(\tau_1,\tau_2)=\frac7{45}, \quad
    x(\tau_1,\tau_3)=-\frac7{40}, \quad
    x(\tau_1,\tau_5)=\frac7{12}, \quad
    x(\tau_1,\tau_8)=\frac7{32},
    $$
    and
    $$
    y(\tau_1,\tau_2)=y(\tau_1,\tau_8)=\frac1{49}, \quad
    y(\tau_1,\tau_3)=y(\tau_1,\tau_5)=\frac{16}{49}.
    $$
  \item[(iii)] Let $d=-480$. There are two $\Gal(H/K)$-orbits
    \begin{equation*}
      \begin{split}
       &\{Q_1=[8,0,15],~Q_2=[24,0,5],~Q_3=[40,0,3],~Q_4=[120,0,1],\\
       &\qquad Q_5=[88,64,13],~Q_6=[104,64,11],~Q_7=[136,128,31],
       ~Q_8=[248,128,17]\} \\
      \end{split}
    \end{equation*}
    and
    \begin{equation*}
      \begin{split}
        &\{Q_1'=[8,8,17],~Q_2'=[136,8,1],~Q_3'=[24,24,11],
        ~Q_4'=[88,24,3], \\
        &\qquad Q_5'=[40,40,13],~Q_6'=[104,40,5],
        ~Q_7=[120,120,31],~Q_8'=[248,120,15]\}.
      \end{split}
    \end{equation*}
    Let $\tau_j$ and $\tau_j'$, $j=1,\ldots,8$, be the corresponding
    CM-points. In the first orbit, the $W$-orbit of $Q_1$ is
    $\{Q_1,Q_4,Q_7,Q_8\}$. We find
    $$
    x(\tau_1,\tau_2)=\frac{11}{240}, \quad
    x(\tau_1,\tau_3)=\frac{31}{320}, \quad
    x(\tau_1,\tau_5)=\frac{11}{16}, \quad
    x(\tau_1,\tau_6)=\frac{31}{96},
    $$
    and
    $$
    y(\tau_1,\tau_2)=y(\tau_1,\tau_5)=\frac1{22^2}, \quad
    y(\tau_1,\tau_3)=y(\tau_1,\tau_6)=\frac1{62^2}.
    $$
    The cases of $\tau_2$ and $\tau_3$ will yield $1/\pi$-series.
    In the second orbit, the $W$-orbit of $Q_1'$ is
    $\{Q_1',Q_2',Q_7',Q_8'\}$. We find that
    $$
    x(\tau_1',\tau_3')=-\frac7{96}, \quad
    x(\tau_1',\tau_4')=\frac{49}{320}, \quad
    x(\tau_1',\tau_5')=-\frac7{16}, \quad
    x(\tau_1',\tau_6')=\frac{49}{240},
    $$
    and
    $$
    y(\tau_1',\tau_3')=y(\tau_1',\tau_5')=\frac1{14^2}, \quad
    y(\tau_1',\tau_4')=y(\tau_1',\tau_6')=\frac1{98^2}.
    $$
    The case of $\tau_3'$ will yield a $1/\pi$-series.
  \end{enumerate}
\end{Example}

\paragraph{\bf Case $\Gamma=\Gamma_0(9)$}

As before, for a negative discriminant $d$, let $K=\Q(\sqrt d)$.
Also, let $R$, $h(d)$, and $H$ be the quadratic order of discriminant
$d$, the class number of $R$, and the ring class field of $R$,
respectively. Let $\CM(d)$ denote the set of CM-points of discriminant
$d$ on $X_0(9)$.

\begin{Lemma} \label{lemma: orbits 9}
  Write a negative discriminant $d$ as $d=r^2d_0$, where
  $d_0$ is a fundamental discriminant. In the following, we represent
  a CM-point by the corresponding quadratic form
  $[A,B,C]:=Ax^2+Bxy+Cy^2$ with $B^2-4AC=d$, $9|A$, and $(A/9,B,C)=1$
  (i.e., $\tau=(-B+\sqrt d)/2A$).
  \begin{enumerate}
  \item[(i)] If $d_0\not\equiv 1\mod 3$ and $3\nmid r$, then
    $|\CM(d)|=0$.
  \item[(ii)] Assume that $d_0\equiv 1\mod 3$ and $3\nmid r$. Let $b$ be
    an integer such that $b^2\equiv d\mod 9$. Then
    $\nu_3(d)=2$ and the two $\Gal(H/K)$-orbits are
    $$
    \{[A,B,C]:B\equiv b\mod 9\}, \qquad
    \{[A,B,C]:B\equiv-b\mod 9\}.
    $$
    The action of $w_9$ and the complex conjugation both swap the two
    orbits. The action of $\SM{-3}{-2}93$ maps all the points in
    $\CM(d)$ to the last orbit in (ii).
  \item[(iii)] Assume that $d_0\equiv 1\mod 3$ and $3\|r$. Then
    $\nu_3(d)=5$ and the five $\Gal(H/K)$-orbits are
    \begin{equation*}
      \begin{split}
        C_1:&\{[A,B,C]:B\equiv 3\mod 9,~3\nmid C\}, \\
        C_2:&\{[A,B,C]:B\equiv 3\mod 9,3|C\}, \\
        C_3:&\{[A,B,C]:B\equiv-3\mod 9,~3\nmid C\}, \\
        C_4:&\{[A,B,C]:B\equiv-3\mod 9,3|C\}, \\
        C_5:&\{[A,B,C]:B\equiv 0\mod 9\}.
      \end{split}
    \end{equation*}
    The action of $w_9$ interchanges $C_1$ with $C_4$, $C_2$ with
    $C_3$, and fixes $C_5$. The action of the complex conjugation
    interchanges $C_1$ with $C_3$, $C_2$ with $C_4$, and fixes $C_5$.
    The action of $\SM{-3}{-2}93$ maps $C_5$ to $\CM(d/9)$ and permutes
    the other CM-points. (It may not map an orbit to an orbit.)
  \item[(iv)] Assume that $d_0\equiv2\mod 3$ and $3\|r$. Then
    $\nu_3(d)=3$ and the three $\Gal(H/K)$-orbits are
    \begin{equation*}
      \begin{split}
        C_1:&\{[A,B,C]: B\equiv 3\mod 9\}, \\
        C_2:&\{[A,B,C]:B\equiv -3\mod 9\}, \\
        C_3:&\{[A,B,C]:B\equiv0\mod 9\}.
      \end{split}
    \end{equation*}
    Both $w_9$ and the complex conjugation switch $C_1$ with $C_2$ and
    fix $C_3$. The action of $\SM{-3}{-2}93$ fixes $C_3$ and permutes
    the other CM-points. (It may not map an orbit to an orbit.)
  \item[(v)] Assume that $27|d$. Then $\nu_3(d)=4$ and the four
    $\Gal(H/K)$-orbits are
    \begin{equation*}
      \begin{split}
        C_1:&\{[A,B,C]:B\equiv 3\mod 9\}, \\
        C_2:&\{[A,B,C]:B\equiv-3\mod 9\}, \\
        C_3:&\{[A,B,C]:B\equiv 0\mod 9,~3\nmid C\}, \\
        C_4:&\{[A,B,C]:B\equiv 0\mod 9,~3|C\}.
      \end{split}
    \end{equation*}
    The action of $w_9$ interchanges $C_1$ with $C_2$ and $C_3$ with
    $C_4$. The action of the complex conjugation interchanges $C_1$
    with $C_2$ and fixes $C_3$ and $C_4$. The action of
    $\SM{-3}{-2}93$ may not map an orbit to an orbit.
  \end{enumerate}
\end{Lemma}

\begin{Remark} The messiness of the action of $\SM{-3}{-2}93$ is due
  to the fact that $\SM{-3}{-2}93$ only normalizes $\Gamma_0(9)$, but
  not $\SM\Z\Z{9\Z}\Z$.
\end{Remark}

\begin{proof} We omit most of the proof since it is very similar to
  that of Lemma \ref{lemma: actions 8}. Here we merely indicate why
  the action of $\SM{-3}{-2}93$ fixes the orbit $C_3$ in (iv).

  Recall that a $\Gal(H/K)$-orbit is determined by its local optimal
  embeddings. In the case of $X_0(9)$, we only need to consider the
  prime $3$. Let $\O_3=\SM{\Z_3}{\Z_3}{9\Z_3}{\Z_3}$. In the
  case $d_0\equiv 2\mod 3$ and $3\|r$, we can check that two optimal
  embeddings of discriminant $d$ of $K$ into $\O_3$ defined by
  $$
  \sqrt d\longmapsto\M{a_1}{b_1}{9c_1}{-a_1}, \quad
  \sqrt d\longmapsto\M{a_2}{b_2}{9c_2}{-a_2}
  $$
  are $\O_3^\times$-equivalent if and only if $a_1\equiv a_2\mod 9$
  and hence $\nu_3(d)=3$. Now assume that $\phi:R\to\O_3$ is an
  optimal embedding of discriminant $d$, say $\phi(\sqrt d)=\SM
  ab{9c}{-a}$ with $a,b,c\in\Z_3$, $\gcd(a,b,c)=1$, and
  $a^2+9bc=d$. We compute that
  $$
  \M{-3}{-2}93\M ab{9c}{-a}\M{-3}{-2}93^{-1}
  =\M{-3a+3b-6c}{-4a/3+b-4c}{6a-9b+9c}{3a-3b+6c}.
  $$
  When $a\equiv 0\mod 9$, since $(a/3)^2+bc=d/9\equiv 2\mod
  3$, we have $b\equiv-c\mod 3$ and hence $-3a+3b-6c\equiv 0\mod 9$.
  This shows that the action of $\SM{-3}{-2}93$ fixes the orbit
  $C_3$.
\end{proof}

When the class group of $R$ is an elementary $2$-group, we can deduce
rationality criteria for $x$ and $y$ analogous to those given in
Corollary \ref{corollary: rationality 8}. However, there are
situations where the class group is not an elementary $2$-group, but
the values of $x$ and $y$ are still rational. It is complicated to
summarize the conditions into a formal statement, so we will just give
one example.

\begin{Example} \label{example: 9}
  Let $d=-1008$. The class group is isomorphic to the
  direct product of a cyclic group of order $4$ and a cyclic group of
  order $2$. Consider the orbit $C_3$ in Part (iv) of Lemma
  \ref{lemma: orbits 9}. It contains
  \begin{equation*}
    \begin{split}
      Q_1=[9,0,28],~Q_2=[36,0,7],&~Q_3=[63,0,4],~Q_4=[252,0,1],\\
      Q_5=[99,90,23],~Q_6=[99,-90,23],&~Q_7=[261,180,32],
      ~Q_8=[261,-180,32].
    \end{split}
  \end{equation*}
  Let $\tau_j$, $j=1,\ldots,8$, be the corresponding CM-points.
  Using the standard cycle notations, the actions of $w_9$,
  $\SM{-3}{-2}93$, and the complex conjugation are
  $(1,4)(2,3)(5,6)(7,8)$, $(1,7)(2,5)(3,6)(4,8)$, and $(5,6)(7,8)$,
  respectively. Moreover, the form class group is generated by $Q_5$
  of order $4$ and $Q_2$ of order $2$ with $Q_4$ being the identity
  element. In the cycle notation, the multiplications by $Q_5$ and by
  $Q_2$ are $(5,1,6,4)(2,8,3,7)$ and $(2,4)(1,3)(5,8)(6,7)$,
  respectively. Let $\sigma_5$ and $\sigma_2$ be their corresponding
  elements in $\Gal(H/K)$.

  Consider $x(\tau_1,\tau_2)$. According the the Shimura reciprocity
  law \cite[Main Theorem II]{Shimura-CM}, if we let $\beta_1$ and
  $\beta_2$ be elements in $\O=\SM\Z\Z{9\Z}\Z$ with positive
  determinants such that
  $$
  \left(99\Z+\M{-45}{-28}{9}{-45}\Z\right)\O=\beta_1\O, \quad
  \left(99\Z+\M{-45}{-7}{36}{-45}\Z\right)\O=\beta_2\O
  $$
  then
  $$
  t(\tau_1)^{\sigma_5}=t(\beta_1^{-1}\tau_1), \qquad
  t(\tau_2)^{\sigma_5}=t(\beta_2^{-1}\tau_2).
  $$
  Here we may choose $\beta_1=\SM9{-1}{18}9$ and
  $\beta_2=\SM{27}8{18}9$ and find that
  $$
  t(\tau_1)^{\sigma_5}=t(\tau_5), \qquad
  t(\tau_2)^{\sigma_5}=t(\tau_7).
  $$
  It follows that
  $$
  x(\tau_1,\tau_2)^{\sigma_5}=x(\tau_5,\tau_7)
  =x\left(\M{-3}{-2}93\tau_2,\M{-3}{-2}93\tau_1\right)
  =x(\tau_2,\tau_1)=x(\tau_1,\tau_2).
  $$
  By a similar computation, we can also show that
  $$
  x(\tau_1,\tau_2)^{\sigma_2}
  =x(\tau_3,\tau_4)=x(w_9\tau_2,w_9\tau_1)
  =x(\tau_2,\tau_1)=x(\tau_1,\tau_2).
  $$
  (Alternatively and slightly more abstractly, if we write $t(\tau_j)$   as $t([Q_j])$ and let $\sigma_j$ denote the element in $\Gal(H/K)$
  corresponding to $\tau_j\in\CM(d)$, then the Shimura reciprocity
  law can be stated as $t([Q_i])^{\sigma_j}=t([Q_j]^{-1}[Q_i])$. From
  this, we may deduce the same conclusion.)

  Since $\sigma_2$ and $\sigma_5$ generate $\Gal(H/K)$ and the value
  of $x(\tau_1,\tau_2)$ is real, we find that $x(\tau_1,\tau_2)\in\Q$
  and so does $y(\tau_1,\tau_2)$. Indeed, we have
  $$
  x(\tau_1,\tau_2)=\frac{52}{675}, \qquad
  y(\tau_1,\tau_2)=\frac1{2704}.
  $$
  In addition to $(\tau_1,\tau_2)$, we may also consider
  $(\tau_1,\tau_3)$ and find
  $$
  x(\tau_1,\tau_3)=\frac{13}{27}, \qquad
  y(\tau_1,\tau_3)=\frac1{2704}.
  $$
\end{Example}

\paragraph{\bf Case $\Gamma=\Gamma_1(5)$}

When the parameters $(a,b,c)$ are $(11,3,-1)$, the modular curve is
$X_1(5)$. In this case, the values of the modular function $t$ at
CM-points lie in the ray class fields of imaginary quadratic orders
(see \cite{Ramachandra}).

As before, for a negative discriminant $d$, let $R$ and $h(d)$
denote the quadratic order of discriminant $d$ and its class number,
respectively. Also, we write $d$ as $d=r^2d_0$, where $d_0$ is a
fundamental discriminant. We let $\CM_0(d)$ and $\CM_1(d)$ denote the
sets of CM-points of discriminant $d$ on $X_0(5)$ and $X_1(5)$.

We first recall that the order of $M(2,\Q)$ in the case of $X_1(5)$
is
$$
\O:=\left\{\M abcd\in M(2,\Z): 5|c,~a\equiv d\mod 5\right\}.
$$
Assume that $\phi$ is an optimal embedding of discriminant $d$ into
$\O$, say,
\begin{equation} \label{equation: embedding 5}
  \phi(\sqrt d)=\M ab{5c}{-a}
\end{equation}
with $a^2+5bc=d$. In order for $\phi(\sqrt d)$ to be in $\O$, the
integer $a$ must be divisible by $5$. It follows that the discriminant
of an optimal embedding into $\O$ must be a multiple of
$5$. Furthermore, if $5|r$ and $\JS{d/25}5=-1$ (respectively,
$\JS{d/25}5=0$), then $6|h(d)$ (respectively, $5|h(d)$) and the values
of $x$ and $y$ at pairs of CM-points of discriminant $d$ will not be
rational numbers. Thus, we only need to consider the case $25\|d$ and
$d/25\equiv\pm1\mod 5$ and the case $5\|d$.

In the case $25\|d$ and $d/25\equiv\pm1\mod 5$, there are two types of
points in $\CM_1(d)$. One is that the integers $b$ and $c$ in
\eqref{equation: embedding 5} are both divisible by $5$ and the other
is that one of $b$ and $c$ is not divisible by $5$. Points of the
first type are mapped to $\CM_0(d/25)$ under the covering $X_1(5)\to
X_0(5)$, while those of the second type are mapped to $\CM_0(d)$.
The number of points of the first type is
$$
\begin{cases}
  2, &\text{if }d=-100, \\
  4h(d/25), &\text{if }d\neq-100, \end{cases}
$$
and the number of points of the second type is $4h(d)=16h(d/25)$.
We see that the values of $x$ and $y$ at pairs of points from
$\CM_1(d)$ can possibly be rational when $h(d/25)\le 2$. However, in
practice, we do not find any such discriminants that yield
rational-valued $x$ and $y$.

In the case $5\|d$, we have $|\CM_1(d)|=2h(d)$. The ray class groups
in general are not elementary $2$-groups even when the class groups
without modulus are elementary $2$-groups. Thus, it is complicated to
state criteria for $x$ and $y$ to have rational values at points
from $\CM_1(d)$, so here we will simply give an example.
(In fact, we find only two discriminants that
yield $1/\pi$-series.)

\begin{Example} Let $d=-760$ with $h(d)=4$ and $|\CM_1(d)|=8$.
  Let $K=\Q(\sqrt d)$, $R$ be the ring of integers in $K$, and $\mathfrak
  p$ be the unique prime ideal of $R$ lying above $5$. Let $G$ and $H$ be
  the ray class group and the ray class field of $R$ of modulus
  $\mathfrak p$, respectively. In terms of the form class group, the
  $8$ elements of the class group are
  \begin{equation*}
    \begin{split}
      Q_1=[5,0,38],~Q_2=[10,0,19],&~Q_3=[95,0,2],~Q_4=[190,0,1], \\
      Q_5=[970,780,157],~Q_6=[515,420,86],&~Q_7=[430,420,103],
      ~Q_8=[785,780,194].
    \end{split}
  \end{equation*}
  Let $\tau_j$, $j=1,\ldots,8$, be the corresponding CM-points.
  The form class group is generated by $Q_1$ of order $4$ and $Q_6$ of
  order $2$. Using the cycle notations, the multiplication by $Q_1$ is
  given by $(1,8,5,4)(2,3,6,7)$ and that by $Q_6$ is
  $(1,7)(2,8)(3,5)(4,6)$. Also, the actions of $w_5$ and
  $\SM2{-1}5{-2}$ are $(1,4)(2,3)(5,8)(6,7)$ and
  $(1,5)(2,6)(3,7)(4,8)$, respectively. Let $\sigma_j$ denote the
  element in $\Gal(H/K)$ corresponding to $\tau_j\in\CM(d)$. Then
  following a similar computation as Example \ref{example: 9}, we can
  show that
  $$
  x(\tau_1,\tau_3)^{\sigma_1}=x(\tau_4,\tau_2)
  =x(w_5\tau_1,w_5\tau_3)=x(\tau_1,\tau_3),
  $$
  $$
  x(\tau_1,\tau_3)^{\sigma_6}=x(\tau_7,\tau_5)
  =x\left(\M{-2}{-1}5{2}\tau_1,\M{-2}{-1}5{2}\tau_3\right)
  =x(\tau_1,\tau_3)
  $$
  and hence $x(\tau_1,\tau_3)\in K$. Since $t(\tau_1)$ and $t(\tau_3)$
  are clearly real, we find that $x(\tau_1,\tau_3)\in\Q$ and so does
  $y(\tau_1,\tau_3)$. Indeed, we find that
  $$
  x(\tau_1,\tau_3)=\frac{19601}{217800}, \qquad
  y(\tau_1,\tau_3)=\frac1{39202^2}.
  $$
\end{Example}


\subsection{Evaluations of the constants in Theorem \ref{theorem:
    general}}

The constants $B_j$ and $C_j$ in Theorem \ref{theorem: general} depend
on the values of several modular functions at CM-points, including
$t_j$, $\epsilon$, and $\delta_j$. Here we shall describe the
strategies to determine their values.
\medskip

\paragraph{\bf Determination of $t_j$}
The evaluation of $t_j$ is the easiest. The value of the modular
function $1/t$ at a CM-point are algebraic integers since $1/t$ is
integral over $\Z[j]$ and the value of the elliptic $j$-function at a
CM-point is known to be an algebraic
integer. Thus, one only needs to evaluate $1/t$ at all
CM-points in the same Galois orbit to sufficient precision and
recognize any symmetric sum of these values as a rational
integer. This will give us the minimal polynomial of $t_j$ over $\Q$
and hence its exact value.
\medskip

\paragraph{\bf Determination of $\epsilon$}
The determination of $\epsilon$ is a little more complicated. Here we
first give an example and explain the general strategy later.

\begin{Example} \label{example: 6 2}
  Consider the case $(a,b,c)=(-17,-6,72)$ with
  $\Gamma=\Gamma_0(6)$,
  \begin{equation} \label{equation: 6 t f}
  t(\tau)=\frac{\eta(2\tau)\eta(6\tau)^5}{\eta(\tau)^5\eta(3\tau)},
  \qquad
  f(\tau)=\frac{\eta(\tau)^6\eta(6\tau)}{\eta(2\tau)^3\eta(3\tau)^2}.
  \end{equation}
  Let $\tau_1$ and $\tau_2$ be the two CM-points of discriminant
  $-420$ given by \eqref{equation: tau 6} in Example \ref{example:
    6}(i). To obtain a $1/\pi$-series using $\tau_1$ and $\tau_2$, we
  are required to determine the value of
  $\epsilon:=f(\tau_2)/f(\tau_1)$.

  Let $R$ be the quadratic order of discriminant $d=-420$. The class
  group of $R$ is generated by $\mathfrak p_2$, $\mathfrak p_3$, and
  $\mathfrak p_5$, the unique prime ideals lying above $2$, $3$, and
  $5$, respectively. Since the quadratic forms associated to $\tau_1$
  and $\tau_2$ are $[6,6,19]$ and $[30,30,11]$, respectively, i.e.,
  the ideals of $R$ associated to $\tau_1$ and $\tau_2$ are $\mathfrak
  p_2\mathfrak p_3$ and $\mathfrak p_2\mathfrak p_3\mathfrak p_5$,
  respectively, there exists a matrix $\gamma$ of determinant $5$ in
  $\SM\Z\Z{6\Z}\Z$ such that $\gamma\tau_1=\gamma_2$. Indeed, we find
  that $\gamma=\SM1{-2}05$ has this property.

  Now let $\gamma_0=\SM5001$ and $\gamma_j=\SM1j05$, $j=1,\ldots,5$,
  be the six matrices defining the Hecke operator $T_5$ on spaces of
  modular forms on $\Gamma_0(6)$ and let $g_j:=f^2|\gamma_j$. Then we
  have
  $$
  \prod_{j=0}^5\left(x-\frac{g_j}{f^2}\right)\in\Q(t)[x].
  $$
  Using Fourier expansions, we can determine this polynomial $P(x)$
  (which is too complicated to be displayed here). We then specialize
  this polynomial $P(x)\in\Q(t)[x]$ at $t=t(\tau_1)$, whose value is
  given in Example \ref{example: 6}(i), and obtain a polynomial $p(x)$
  in $x$ of degree $6$ over $\Q(\sqrt3,\sqrt5,\sqrt7)$. Since
  $\tau_2=\SM1{-2}05\tau_1$,
  $f(\tau_2)^2/5f(\tau_1)^2=g_3(\tau_1)/f(\tau_1)^2$ is
  a root of $p(x)$. Observe that $\gamma_j\tau_1$, $j\neq 3$,
  are CM-points of discriminant $-5^2\cdot420$. This means that among
  the six roots of $p(x)$, five are in the ring class field of the
  quadratic order of discriminant $-5^2\cdot420$ and the remaining
  root is in $\Q(\sqrt3,\sqrt5,\sqrt7)$. This is the value of
  $f(\tau_2)^2/5f(\tau_1)^2$. After completing these tedious
  calculations, we find that
  $$
  \frac{f(\tau_2)}{f(\tau_1)}
  =\frac{\sqrt5}2(-5+3\sqrt3)(3+\sqrt7)(4-\sqrt{15})(6+\sqrt{35}).
  $$
  (Notice that $f(\tau_2)/\sqrt5f(\tau_1)$ is a unit in
  $\Q(\sqrt3,\sqrt5,\sqrt7)$. There should be a theoretical
  explanation for this phenomenon, but we will not pursue it here.)
\end{Example}

From the example, it should be evident how to determine
$\epsilon=f(\tau_2)/f(\tau_1)$. Namely, we find a matrix $\gamma$ in
the relevant order ($\SM\Z\Z{N\Z}\Z$ except for the case
$(a,b,c)=(11,3,-1)$) of $M(2,\Q)$. Let $n=\det\gamma$ and $\gamma_j$
be the matrices defining the Hecke operators $T_n$, express the
polynomial
$$
\prod\left(x-\frac{f^2|\gamma_j}{f^2}\right)
$$
as a polynomial $P(x)$ in $\Q(t)[x]$, and specialize $P(x)$ at
$t=t(\tau_1)$, say $p(x)$. Then $f(\tau_2)/f(\tau_1)$ will be the
unique root of $p(x)$ that is in the same field as $t(\tau_1)$.
\medskip

\paragraph{\bf Determination of $\delta_j$}

The determination of $\delta_j$ is the most complicated
part. Traditionally, it is done by using modular equations and known
values of modular functions at CM-points. Here we use a different
approach.

\begin{Lemma} Let $t(\tau)$ and $f(\tau)$ be the modular function and
  the modular form of weight $1$ associated to one of the sporadic
  Ap\'ery-like sequences, as given in Table \eqref{table: sporadic}.
  Let $\tau_0$ be a CM-point of discriminat $d$ with $\phi$ being the
  corresponding optimal embedding. Let $\alpha=\phi(\sqrt d)$ and set
  $$
  g=\theta_qt,\qquad h=\frac{g|_2\alpha}g, \qquad
  \delta=\frac{\theta_qh}{f^2}(\tau_0).
  $$
  Then there exists an explicitly computable positive integer $A$,
  depending only on the Galois orbit over $\Q$ of the given CM-point,
  such that $A\delta$ is an algebraic integer.
\end{Lemma}

\begin{Example} Let us continue working on the case considered in
  Example \ref{example: 6 2}, i.e., let $t$ and $f$ be given by
  \eqref{equation: 6 t f}. We check that
  $$
  g=t(1+17t+72t^2)f^2.
  $$
  (In fact, $\theta_qt=t(1-at+ct^2)f^2$ holds in all six
  cases.) Note that
  \begin{equation} \label{equation: div 1}
    \begin{split}
  \div t=(\infty)-(0), \qquad
  &\div(1+8t)=(1/3)-(0), \\
  \div(1+9t)=(1/2)-(0), \qquad
  &\div f^2=2(0),
    \end{split}
  \end{equation}
  and hence
  \begin{equation} \label{equation: div 2}
  \div g=(\infty)+(1/2)+(1/3)-(0).
  \end{equation}

  Consider $\tau_1=(-6+\sqrt{-420})/12$, we have
  $\alpha=\SM{-3}{-19}63$. Observe that
  $$
  \alpha=\M{-3}{-2}63\M1{-17}0{35}.
  $$
  Let $w_3=\SM{-3}{-2}63$ and $\gamma=\SM1{-17}0{35}$. We check that
  $$
  t\big|w_3=-\frac{1+9t}{9(1+8t)}, \qquad
  f^2\big|w_3=9\frac{\eta(3\tau)^{12}\eta(2\tau)^2}
  {\eta(6\tau)^6\eta(\tau)^4}
  =9(1+8t)^2f.
  $$
  Thus,
  $$
  g\big|\alpha=(g\big|w_3)\big|\gamma
  =\frac{g}{9(1+8t)^2}\Big|\gamma.
  $$
  Note that
  \begin{equation} \label{equation: div 3}
  \div\frac g{(1+8t)^2}=(\infty)+(1/2)+(0)-(1/3).
  \end{equation}
  Now let $\gamma_{e,j}=\SM ej0{35/e}$, $e=1,5,7,35$ and
  $j=0,\ldots,35/e-1$, be the standard coset representatives defining
  the Hecke operator $T_{35}$ on spaces of modular forms on $X_0(6)$
  and let
  $$
  h_{e,j}=\frac{g/(1+8t)^2|\gamma_{e,j}}{9g}
  $$
  (with $h=h_{1,18}$ being one of them).
  Since
  $$
  \frac g{(1+8t)^2}\Big|\gamma_{e,j}
  =\frac e{35}\frac{g((e^2\tau+ej)/35)}
  {(1+8t((e^2\tau+ej)/35))^2},
  $$
  we find that the Fourier coefficients of
  $105^2(\theta_qh_{e,j})/f^2$ are all algebraic integers and that any
  symmetric sum of $105^2(\theta_qh_{e,j})/f^2$ has
  integer Fouier coefficients. Futhermore, $\gamma_{e,j}\infty$
  (respectively, $\gamma_{e,j}(1/2)$, $\gamma_{e,j}(1/3)$,
  $\gamma_{e,j}0$) are all equivalent to $\infty$ (respectively,
  $1/2$, $1/3$, $0$) under $\Gamma_0(6)$ for all $e$ and $j$. Hence,
  by \eqref{equation: div 1}, \eqref{equation: div 2}, and
  \eqref{equation: div 3}, we have
  $$
  \div\sum_{e,j}\left(\frac{\theta_qh_{e,j}}{f^2}\right)^k
  \ge-k((\infty)+(1/2)+2(1/3)+2(0))
  $$
  for any positive integer $k$. Thus, if we let $j_6=1/t$, which has
  values $\infty$, $-9$, $-8$, and $0$ at the cusps $\infty$, $1/2$,
  $1/3$, and $0$, respectively, then
  $$
  \sum_{e,j}\left(105^2j_6^2(j_6+9)(j_6+8)^2
    \frac{\theta_qh_{e,j}}{f^2}\right)^k
  $$
  is a modular function on $\Gamma_0(6)$ with integer Fourier
  coefficients and a unique pole at $\infty$ and therefore is equal to
  $P_k(j_6)$ for some $P_k(x)\in\Z[x]$. Now the value of $j_6$ at the
  CM-point $\tau_1$ is an algebraic integer, which implies that
  $$
  105^2j_6(\tau_1)^2(j_6(\tau_1)+9)(j_6(\tau_1)+8)^2
  \frac{\theta_qh_{e,j}}{f^2}(\tau_1)
  $$
  are all algebraic integers for all $e$ and $j$. In particular, if we
  write the value of $j_6(\tau_1)^2(j_6(\tau_1)+9)(j_6(\tau_1)+8)^2$ as
  $B/C$ for some rational integer $B$ and some algebraic integer $C$,
  then $A\delta$ is an algebraic integer for $A=105^2B$. Here we find
  that
  \begin{equation*}
    \begin{split}
  &\frac1{j_6(\tau_1)^2(j_6(\tau_1)+9)(j_6(\tau_1)+8)^2}
   =\frac1{1728}(2-\sqrt3)^6\left(\frac{1-\sqrt5}2\right)^{18} \\
   &\qquad\qquad\qquad\times(4-\sqrt{15})^3
   \left(\frac{5-\sqrt{21}}2\right)^6
  (6-\sqrt{35})^3(2\sqrt5-\sqrt{21})^3.
  \end{split}
  \end{equation*}
  That is, we may choose $B=1728$ and $A=105^2\cdot1728$. From the
  argument, we see that the same integer $A$ works for all
  CM-points in the same Galois orbit of $\tau_1$. Thus, we can
  determine the value of $\delta$ by evaluating
  numerically $A\delta$ to sufficient precision for all CM-points in
  the Galois orbit and recognize every symmetric sum as a rational
  integer. After a tedious computation, we arrive at
  \begin{equation*}
    \begin{split}
      \delta&=\frac1{3\sqrt{35}}(2-\sqrt3)^2
       \left(\frac{1+\sqrt5}2\right)^6(8+3\sqrt7)(4-\sqrt{15})
       \left(\frac{5-\sqrt{21}}2\right)^2(6+\sqrt{35})
       \\
       & (\sqrt{21}-2\sqrt5) \left(\frac{87+80\sqrt3+84\sqrt5+48\sqrt7+14\sqrt{15}
         +17\sqrt{21}-34\sqrt{35}-4\sqrt{105}}2\right).
    \end{split}
  \end{equation*}
  The value of $\delta$ for $\tau_2=(-30+\sqrt{-420})/60$ is the
  Galois conjugate of the above number obtained by $\sqrt3\to-\sqrt3$,
  $\sqrt5\to\sqrt5$, and $\sqrt7\to-\sqrt7$. Combining this
  computation with the results from Examples \ref{example: 6}(i) and
  \ref{example: 6 2} and applying Theorem \ref{theorem: general}, we
  obtain
  $$
  \sum_{n=0}^\infty\sum_{m=0}^{\lfloor n/2\rfloor}
  u_n\binom{2m}m\binom n{2m}(5n-2)\left(-\frac{71}{1008}\right)^n
  \left(\frac1{142^2}\right)^m=\frac{3\sqrt{35}}\pi,
  $$
  where $\{u_n\}$ is the sequence corresponding to
  $(a,b,c)=(-17,-6,72)$ in Table \ref{table: sporadic}.
\end{Example}

\subsection*{Acknowledgements}
The first author was partially supported by the National Natural Science Foundation of China (11801424) and a start-up research grant of the
Wuhan University. The second author was partially supported by Grant 106-2115-M-002-009-MY3 of the Ministry of Science and Technology,
Taiwan (R.O.C.).

The authors would like to thank Wadim Zudilin for many insightful
comments on the earlier draft of the paper, especially those about
Brafman's works.

\appendix
\section{$2$-variable $1/\pi$-series for sporadic Ap\'ery sequences}
\label{section: sporadic data}

\begin{table}[!htbp]
$$ \extrarowheight3pt
\begin{array}{c|cccccc} \hline\hline
  d & x & y & A & B & C & \text{Ref.}\\ \hline
  -96 & -1/12 & -1/8 & 9 & 3 & 4\sqrt2 &\text{\cite[(6.1)]{Sun-2019}}
  \\
  -192 & -1/36 & -1/32 & 99 & 23
                    &  39\sqrt2&\text{\cite[(6.2)]{Sun-2019}} \\
  -240 & 1/60 & 1/64 & 315 & 51 & 92\sqrt5
              & \text{\cite[(6.3)]{Sun-2019}} \\
   & 47/441 & 1/47^2 & 180 & -8 & 483\sqrt5 & \text{\cite[(6.4)]{Sun-2019}}\\
  -660 & -97/1176 & 1/194^2 & 495 & 138 & 112\sqrt 5  &\text{\cite[(6.5)]{Sun-2019}}\\
  -840 & 241/5808 & 1/482^2 & 630 & 75 & 374\sqrt2 &\text{\cite[(6.6)]{Sun-2019}} \\
    & 449/5600 & 1/898^2 & 990 & 31 & 680\sqrt7  &\text{\cite[(6.7)]{Sun-2019}}\\
  -1092 & -727/6776 & 1/1454^2 & 585 & 172 & 110\sqrt7  &\text{\cite[(6.8)]{Sun-2019}}\\
  -1320 & 1057/50784 & 1/2114^2 & 630 & 77 & 92\sqrt{15} &\text{\cite[(6.9)]{Sun-2019}} \\
    & 4801/47040 & 1/9602^2 & 2277 & -412 & 2156\sqrt{15} &  \\
  -1380 & -1151/52992 & 1/2302^2 & 94185 & 17014 &
    8520\sqrt{23} &\text{\cite[(6.10)]{Sun-2019}}\\
  -1428 & -2177/30976 & 1/4354^2 & 5355 & 1381 & 968\sqrt7 &\text{\cite[(6.11)]{Sun-2019}}\\
  -1848 & 8449/237952 & 1/16898^2 & 62370 & 6831 & 2912\sqrt{231} &\text{\cite[(6.12)]{Sun-2019}}\\
    & 19601/226800 & 1/39202^2 & 12870 & -439 & 6930\sqrt{14}\\
\hline\hline
\end{array}
$$
\caption{$1/\pi$-series for $(a,b,c)=(7,2,-8)$}
\label{table: 6-1}
\end{table}

\begin{table}[!htbp]
$$ \extrarowheight3pt
\begin{array}{c|cccccc} \hline\hline
  d & x & y & A & B & C &\text{Ref.}\\ \hline
  -96 & -7/81 & -8/7^2 & 32 & 12 & 9\sqrt3
      &\text{\cite[(7.1)]{Sun-2019}} \\
  -192 & -31/1089 & -32/31^2 & 64 & 16 & 33
                           &\text{\cite[(7.2)]{Sun-2019}}\\
  -240 & 13/135 & 1/52^2 & 7 & -1 & 30\sqrt3
                         &\text{\cite[(7.3)]{Sun-2019}}\\
    & 65/3969 & 64/65^2 & 160 & 24 & 63\sqrt3
                         &\text{\cite[(7.4)]{Sun-2019}} \\
  -660 & -89/990 & 1/178^2 & 280 & 93 & 20\sqrt{33}  &\text{\cite[(7.5)]{Sun-2019}}\\
  -840 & 251/6300 & 1/502^2 & 176 & 15 & 25\sqrt{42}  &\text{\cite[(7.6)]{Sun-2019}}\\
       & 485/6534 & 1/970^2 & 560 & -23 & 693\sqrt3 &\text{\cite[(7.7)]{Sun-2019}}\\
 -1320 & 1079/52920 & 1/2158^2 & 12880 & 1353 & 4410\sqrt3  &\text{\cite[(7.8)]{Sun-2019}}\\
    & 5291/57132 & 1/10582^2 & 6160 & -1824 & 15939\sqrt3
                 & \text{\cite[(7.11)]{Sun-2019}} \\
 -1380 & -1126/50715 & 1/2262^2 & 19136 & 3776 & 735\sqrt{115} &\text{\cite[(7.9)]{Sun-2019}}\\
 -1428 & -2024/26775 & 1/4048^2 & 24640 & 7552 & 2415\sqrt{17} &\text{\cite[(7.10)]{Sun-2019}}\\
 -1848 & 8749/255150 & 1/17498^2 & 32032 & 2546 & 14175\sqrt3 &\text{\cite[(7.12)]{Sun-2019}} \\
       & 21295/267696 & 1/42590^2 & 560 & -67 & 286\sqrt{22} \\
\hline\hline
\end{array}
$$
\caption{$1/\pi$-series for $(a,b,c)=(10,3,9)$}
\label{table: 6-2}
\end{table}

\begin{table}[!htbp]
$$ \extrarowheight3pt
\begin{array}{c|ccccc} \hline\hline
  d & x & y & A & B & C \\ \hline
  -96 & -1/18 & -1/32 & 1 & 0 & \sqrt6\\
  -180 & -1/16 & 1/18^2 & 5 & -1 & 9\sqrt3 \\
    & -31/320 & 1/62^2 & 5 & -6 & 30\sqrt3 \\
  -192 & -5/216 & -1/50 & 11 & 1 & 6\sqrt3\\
  -240 & 7/360 & 1/7^2 & 35 & 9 & 4\sqrt{15} \\
  -420 & -71/1008 & 1/142^2 & 5 & -2 & 3\sqrt{35} \\
    & -161/1728 & 1/322^2 & 35 & -41 & 252\sqrt3\\
  -660 & -161/3240 & 1/322^2 & 385 & -57 & 360\sqrt3\\
    & -1079/10584 & 1/2158^2 & 5 & -11 & 84\sqrt3 \\
  -840 & 161/2592 & 1/322^2 & 385 & 150 & 72\sqrt 3\\
  -1092 & -1351/23400 & 1/2702^2 & 385 & -101 & 210\sqrt{13} \\
    & -6049/60984 & 1/12098^2 & 65 & -124 & 132\sqrt{91}\\
  -1320 & 881/35280 & 1/1762^2 & 805 & 209 & 210\sqrt2 \\
  -1380 & -1351/73008 & 1/2702^2 & 1610 & 76 & 585\sqrt3 \\
    & -51841/476928 & 1/103682^2 & 65 & -375 & 4968\sqrt3 \\
  -1428 & -3401/75600 & 1/6802^2 & 13090 & -1714 & 9765\sqrt3 \\
    & -28799/278784 & 1/57598^2 & 595 & -1602 & 2728\sqrt{51} \\
  -1848 & 6049/121968 & 1/12098^2 & 30030 & 10506 & 4081\sqrt6\\
  \hline\hline
\end{array}
$$
\caption{$1/\pi$-series for $(a,b,c)=(-17,-6,72)$}
\label{table: 6-3}
\end{table}

\begin{table}[!htbp]
$$ \extrarowheight3pt
\begin{array}{c|cccccc} \hline\hline
  d & x & y & A & B & C & \text{Ref.} \\ \hline
  -180 & -2/27 & -1/4^2 & 26 & 5 & 27\sqrt3 \\
  -288 & -7/50 & -1/14^2 & 13 & 3 & 30\sqrt2
                        &\text{\cite[(9.2)]{Sun-2019}}\\
  -315 & -1/6 & 1/18^2 & 14 & 5 & 27\sqrt3
                        &\text{\cite[(9.3)]{Sun-2019}} \\
  -576 & -22/243 & -1/968 & 76 & 8 & 81\sqrt3
                 &\text{\cite[(9.4)]{Sun-2019}} \\
  -819 & -55/378 & 1/110^2 & 182 & 37 & 315\sqrt3  &\text{\cite[(9.6)]{Sun-2019}}\\
  -1008 & 52/675 & 1/52^2 & 182 & 64 & 45\sqrt3 &\text{\cite[(9.5)]{Sun-2019}} \\
  -3627 & -12151/95256 & 1/24302^2 & 107198 & 8989 & 147420\sqrt3 &\\
  -3843 & -6049/151200 & 1/12098^2 & 410774 & 33451 & 182700\sqrt3 &\\
  \hline\hline
\end{array}
$$
\caption{$1/\pi$-series for $(a,b,c)=(-9,-3,27)$}
\label{table: 9}
\end{table}

\begin{table}[!htbp]
$$ \extrarowheight3pt
\begin{array}{c|cccccc} \hline\hline
  d & x & y & A & B & C & \text{Ref.} \\ \hline
  -240 & 13/225 & 1/52^2 & 145 & 9 & 285 &\text{\cite[(8.1)]{Sun-2019}} \\
  -760 & 19601/217800 & 1/39202^2 & 95 & -1388 & 9405\sqrt{19} &\\
  \hline\hline
\end{array}
$$
\caption{$1/\pi$-series for $(a,b,c)=(11,3,-1)$}
\label{table: 5}
\end{table}

\begin{table}[!htbp]
$$ \extrarowheight3pt
\begin{array}{c|cccccc} \hline\hline
  d & x & y & A & B & C & \text{Ref.}\\ \hline
  -60 & 1/32 & 1 & 15 & 3 & 8(2+\sqrt5) & \text{\cite[(5.1)]{Sun-2019}} \\
  -64 & 1/36 & -2 & 3 & 1 & 3 &\text{\cite[(5.3)]{Sun-2019}}\\
  -32,-288 & 7/96 & 1/14^2 & 3 & 0 & 8 &\text{\cite[(5.4)]{Sun-2019}}\\
  -112 & -1/252 & 16 & 21 & 5 & 6\sqrt 7 &\text{\cite[(5.2)]{Sun-2019}}\\
  -480 & -7/96 & 1/14^2 & 30 & 11 & 12  &\text{\cite[(5.5)]{Sun-2019}}\\
    & 11/240 & 1/22^2 & 15 & 1 & 6\sqrt{10} &\text{\cite[(5.6)]{Sun-2019}} \\
    & 31/320 & 1/62^2 & 30 & -7 & 160 &\text{\cite[(5.9)]{Sun-2019}}\\
  -672 & -13/336 & 1/26^2 & 21 & 6 & 2\sqrt{21} &\text{\cite[(5.7)]{Sun-2019}} \\
    & 17/576 & 1/34^2 & 21 & 2 & 18  &\text{\cite[(5.8)]{Sun-2019}}\\
    & 97/896 & 1/194^2 & 6 & -3 & 56  &\text{\cite[(5.13)]{Sun-2019}}\\
  -1120 & -71/720 & 1/142^2 & 210 & 85 & 33\sqrt5 &\text{\cite[(5.12)]{Sun-2019}}\\
    & 127/2304 & 1/254^2 & 210 & -1 & 288 &\text{\cite[(5.16)]{Sun-2019}}\\
    & 251/2800 & 1/502^2 & 42 & -10 & 105\sqrt{2} &\text{\cite[(5.17)]{Sun-2019}}\\
  -1248 & -49/4800 & 1/98^2 & 195 & 34 & 80 &\text{\cite[(5.10)]{Sun-2019}}\\
    & 53/5616 & 1/106^2 & 195 & 22 & 27\sqrt{13} &\text{\cite[(5.11)]{Sun-2019}}\\
    & 1249/10400 & 1/2498^2 & 78 & -131 & 2600 &\text{\cite[(5.20)]{Sun-2019}}\\
  -1632 & -97/18816 & 1/194^2 & 1785 & 254 & 672 &\text{\cite[(5.14)]{Sun-2019}}\\
    & 101/20400 & 1/202^2 & 210 & 23 & 15\sqrt{34} &\text{\cite[(5.15)]{Sun-2019}}\\
    & 4801/39200 & 1/9602^2 & 510 & -1523 & 33320 &\text{\cite[(5.23)]{Sun-2019}}\\
  -2080 & -577/18496 & 1/1154^2 & 7410 & 1849 & 2992 &\text{\cite[(5.18)]{Sun-2019}}\\
    & 721/28880 & 1/1442^2 & 6630 & 505 & 2014\sqrt5 &\text{\cite[(5.19)]{Sun-2019}}\\
    & 5201/46800 & 1/10402^2 & 570 & -457 & 1590\sqrt{13}&\text{\cite[(5.24)]{Sun-2019}} \\
  -3040 & -2737/197136 & 1/5475^2 & 62985 & 11363 & 7659\sqrt{10} &\text{\cite[(5.21)]{Sun-2019}} \\
    & 3041/243360 & 1/6082^2 & 358530 & 33883 & 176280 &\text{\cite[(5.22)]{Sun-2019}}\\
    & 52021/439280 & 1/104042^2 & 3705 & -5918 & 36499\sqrt5 \\
    \hline\hline
\end{array}
$$
\caption{$1/\pi$-series for $(a,b,c)=(12,4,32)$}
\label{table: 8}
\end{table}

\section{2-Variable $1/\pi$-series for hypergeometric sequences}\label{section:hyper-data}
\begin{small}
\begin{table}[!htbp]
$$ \extrarowheight3pt
\begin{array}{c|cccccc} \hline\hline
  d & x & y & A & B & C & \text{Ref.}\\ \hline
 -7 & -{1}/{16} & 16 & 30 & 7 & 24  & \text{\cite[(\uppercase\expandafter{\romannumeral1}1)]{Sun-list}} \\
 -16, -64 &1/2 & -1/32&  6 & 1 & 2\sqrt{8+6\sqrt{2}} &\text{\cite[(\uppercase\expandafter{\romannumeral1}5)]{Sun-2019}}^{\rm a}  \\
   -192 & -{17}/{32} & {1}/{34^2} & 30 & 7 & 12 & \text{\cite[(\uppercase\expandafter{\romannumeral1}2)]{Sun-list}}\\
   -240 &  {31}/{128}  & {1}/{62^2}   & 42 & 5 & 16\sqrt{3} & \text{\cite[(\uppercase\expandafter{\romannumeral1}4)]{Sun-list}}\\
     &  {97}/{128}  & {1}/{194^2}  &  30 & -1 & 80 &  \text{\cite[(\uppercase\expandafter{\romannumeral1}3)]{Sun-list}} \\
 \hline\hline
\end{array}
$$
$^{\rm a}$ \footnotesize{For this case we take $\tau_1=\frac{1}{2}+\frac{1}{2}i$ and $\tau_2=i$ in Remark \ref{rem-thm-Sun}.}
\caption{$1/\pi$-series for $a=1/2$}
\label{table:a=2}
\end{table}
\vspace{-0.5cm}
\begin{table}[!htbp]
$$ \extrarowheight3pt
\begin{array}{c|cccccc} \hline\hline
  d & x & y & A & B & C & \text{Ref.}\\ \hline
 -12,-48&  {1}/{2} & {1}/54 & 60 & 8 & {45\sqrt{3}}  & \text{ \cite[(\uppercase\expandafter{\romannumeral2}1)]{Sun-list}} \\
  -20 & -7/{20} & 729/980 & 54 & 12 & {15\sqrt{3}+\sqrt{15}}  & \text{\cite[(\uppercase\expandafter{\romannumeral2}11)]{Sun-list}} \\
  -32 & -{2}/{25} & 729/800  & 756 & 132 & 75\sqrt{3}+100\sqrt{6} & \text{\cite[(\uppercase\expandafter{\romannumeral2}12)]{Sun-list}}\\
  -35 & {13}/{256} & 729/676 & 135 & 21 & 24\sqrt{3}+8\sqrt{15} & \text{\cite[(\uppercase\expandafter{\romannumeral2}10)]{Sun-list}} \\
  -60 &-3/5 & -1/45 &  72 & 16 & 7\sqrt{15} & \text{\cite[(\uppercase\expandafter{\romannumeral2}13)]{Sun-2019}} \\
  -72 & {27}/{100} & {1}/100 & 182 & 24 & 75\sqrt{3} & \text{\cite[(\uppercase\expandafter{\romannumeral2}2)]{Sun-list}} \\
      & 73/100 & 729/730^2 & 18 & 1 & 25\sqrt{3} & \text{\cite[(\uppercase\expandafter{\romannumeral2}5)]{Sun-list}} \\
  -84 & -1/{4} & -1/108 & 39 & 7 &9\sqrt{3} & \text{\cite[(\uppercase\expandafter{\romannumeral2}14)]{Sun-2019}} \\
  -120 & 1/{12} & 1/{18}^2  &210 & 25 & 54\sqrt{3}  & \\
     & 11/12 & 1/{{198}^2} & 30 & -8 & 135\sqrt{3} & \text{\cite[(\uppercase\expandafter{\romannumeral2}3)]{Sun-list}}\\
  -132 & -{3}/{44} & -{1}/396 & 45 & 6 &5\sqrt{11}  &\\
  -168 & {3}/{100} & {1}/{{30}^2} & 198  & 21 & 50\sqrt{2}&\\
    & {97}/{100} & {1}/{{970}^2} & 42 & -41 &525\sqrt{3}  & \text{\cite[(\uppercase\expandafter{\romannumeral2}4)]{Sun-list}}\\
  -228 & -{1}/{100} & -{1}/2700 & 17157 & 1654 & 2925\sqrt{3}  &\\
  -240 & 488/1331 & 1/488^2 & 11310 & 976 & 4719\sqrt{3}&\\
       & {843}/1331 & {1}/{843^2} & 2520 & 48 & 1573\sqrt{5} &\\
  -312 & {3}/1156 & {1}/{102^2} &13860 & 1118 & 1445\sqrt{6}    & \text{\cite[(\uppercase\expandafter{\romannumeral2}6)]{Sun-list}}\\
    & {1153}/1156 & {1}/{39202^2}  &390 & -3967 & 56355\sqrt{3} & \text{\cite[(\uppercase\expandafter{\romannumeral2}8)]{Sun-list}} \\
  -372 &-1/900 & -{1}/24300 & 105339 & {7843} & 14175\sqrt{3} &\\
   -408 & {1}/1452 &{1}/{198^2} & 888420 &62896 & 114345\sqrt{3} & \text{\cite[(\uppercase\expandafter{\romannumeral2}7)]{Sun-list}} \\
    & {1451}/1452 & {1}/{{287298}^2} & 210 &-7157 & 114345\sqrt{3} & \text{\cite[(\uppercase\expandafter{\romannumeral2}9)]{Sun-list}} \\
  -435 & {-107}/256 &{1}/{{5778}^2} & 39585 &7075 & 7344\sqrt{3} & \\
  -555 & -{9249}/42592 & {1}/{{18498}^2} & 7245 & 1073 & 605\sqrt{15} &\\
  -708 & -{3}/124844 &-{1}/{1123596} & 6367095 & 342786 & 140185\sqrt{59} & \\
  -795 & -{7361}/96000 & {1}/{{132498}^2} & 62403 & 7049 &10800\sqrt{3} & \\
 \hline\hline
\end{array}
$$
\caption{$1/\pi$-series for $a=1/3$}
\label{table:a=3}
\end{table}

\begin{table}[!htbp]
$$ \extrarowheight3pt
\begin{array}{c|cccccc} \hline\hline
  d & x & y & A & B & C & \text{Ref.}\\ \hline
  -8,-72&1/2 &1/98^2&360&27&70\sqrt{21}&
\text{ \cite[(\uppercase\expandafter{\romannumeral3}3)]{Sun-list}} \\
  -15 & -7/{21}^2& 64^2/7^2 & 640 & 104 & 21\sqrt{42}+14\sqrt{210}  & \text{ \cite[(\uppercase\expandafter{\romannumeral3}5)]{Sun-list}} \\
  -48 & 257/{33}^2 & {16}^2/{257}^2 & 160 & 18 &11\sqrt{66} & \text{\cite[(\uppercase\expandafter{\romannumeral3}1)]{Sun-list}} \\
    & 832/{33}^2& 1/{52^2} & 85 & 2 &33\sqrt{33} &  \text{\cite[(\uppercase\expandafter{\romannumeral3}4)]{Sun-list}}\\
  -64 & -511/{63}^2 & -{512}/{{511}^2} & 704 & 92 &{63\sqrt{14}} & \\
 -84 & -{110}/{12^2} & {1}/{110^2} & 28 & 5 & 3\sqrt{6}  & \text{\cite[(\uppercase\expandafter{\romannumeral3}2)]{Sun-list} }\\
  -112 & {4097}/{513^2} & {64^2}/{4097^2} & 19040  & 1682 & {513\sqrt{114}} &\\
   &{259072}/{513^2} & {1}/{4048^2} & 455 & -784 & 2052\sqrt{57} & \\
  -120 & {322}/{42^2} & {1}/{322^2} & 760 & 71 &126\sqrt{7} & \text{\cite[(\uppercase\expandafter{\romannumeral3}6)]{Sun-list}}\\
       &{1442}/{42^2} & {1}/{1442^2} &40 & -4 &7\sqrt{210}  & \text{\cite[(\uppercase\expandafter{\romannumeral3}7)]{Sun-list}}\\
  -132 & -{398}/{48^2} &1/{398^2} & 260 & 33 &32\sqrt{6} & \text{\cite[(46)]{CWZ}}\\
  -168 &898/114^2 & {1}/{898^2} &3080 &276 &95\sqrt{114} & \text{\cite[(\uppercase\expandafter{\romannumeral3}8)]{Sun-list}}\\
       & {12098}/{114^2} &{1}/{12098^2} & 280 & -139 &95\sqrt{399} & \text{\cite[(\uppercase\expandafter{\romannumeral3}9)]{Sun-list}} \\
  -228 & -{2702}/{336^2} &{1}/{2702^2} & 83980 & 7331 & 3360\sqrt{42}  & \text{\cite[(47)]{CWZ}} \\
  -280 & {39202}/{378^2} & {1}/{2702^2} & 1840 & 136 & 135\sqrt{42} & \\
       & {103682}/{{378}^2} &{1}/{{103682}^2} & 440 & -25 &  378\sqrt{3} &\\
  -312 & {10402}/{1302^2} & {1}/{10402^2} & 337480 & 24044 & 3689\sqrt{434}  & \text{\cite[(\uppercase\expandafter{\romannumeral3}10)]{Sun-list}}\\
      & {1684802}/{1302^2} &1/{1684802^2} & 8840 & -50087 & 7378\sqrt{8463}  & \text{\cite[(\uppercase\expandafter{\romannumeral3}11)]{Sun-list}}\\
  -340 & -{103682}/{684^2} & {1}/{103682^2}& 97580 & 12197 & 2736\sqrt{95} &\\
  -372 & -24302/3036^2 & {1}/{24302^2} & 2143960 & 142322 &11385\sqrt{1518}  & \text{\cite[(48)]{CWZ}}\\
  -408 & {39202}/{4902^2} & {1}/{39202^2} &11657240 & 732103 & 80883\sqrt{817}  & \text{\cite[(\uppercase\expandafter{\romannumeral3}12)]{Sun-list}}\\
     & {23990402}/{4902^2} & {1}/{23990402^2} &3080 & -58871 &17974\sqrt{2451} & \text{\cite[(\uppercase\expandafter{\romannumeral3}13)]{Sun-list}}\\
  -520 & {1684802}/{5922^2} & 1/{1684802^2} & 1382440 & 106756 &14805\sqrt{658}  &  \\
      & 33385282/5922^2 & 1/{33385282^2} & 337480& -320300& 115479\sqrt{658} &\\
    -532 &-5177198/{3348^2} & {1}/{5177198^2} & 602140 & 88597 & 4185\sqrt{1302} & \\
     -708 & {1123598}/{140448^2} &{1}/{1123598^2} & 1898208780 &90848259 & 4307072\sqrt{4389} & \text{\cite[(49)]{CWZ}}\\
  -760 &{33385282}/{55062^2} & {1}/{33385282^2} & 27724840 & {1877581}  & 49266\sqrt{15295} &\\
   & 2998438562/{55062^2} & {1}/{2998438562^2}& 72760 & -289964&  156009\sqrt{322} &\\
  \hline\hline
\end{array}
$$
\caption{$1/\pi$-series for $a=1/4$}
\label{table:a=4}
\end{table}

\end{small}

\end{document}